\documentclass[12pt]{article}
\usepackage{aurical}
\usepackage[T1]{fontenc}
\usepackage{amssymb,amsmath,bm}
\usepackage{amsthm}
\usepackage{setspace,mathrsfs,nicefrac}
\usepackage{bbm,url,color,wasysym}
\usepackage[margin=1in]{geometry}
\usepackage{paralist}\setdefaultitem{$-$}{\tiny $\bullet$}{\tiny$\circ$}{$\star}
\usepackage{hyperref}
\usepackage[title]{appendix}

\allowdisplaybreaks

\renewcommand{\geq}{\geqslant}
\renewcommand{\leq}{\leqslant}

\author{{\bf Robert C. Dalang$^*$}  and {\bf Fei Pu\footnote{Research
partially supported by the Swiss National Foundation for Scientific
Research.}}\\
\\
\it\small \'Ecole Polytechnique F\'ed\'erale de Lausanne \\
}

\title{\bf\Large{Optimal lower bounds on hitting probabilities for stochastic heat equations in spatial dimension $k \geq 1$}}
\date{}

\newtheorem{stat}{Statement}[section]
\newtheorem{prop}[stat]{Proposition}

\newtheorem{theorem}[stat]{Theorem}

\newtheorem{lemma}[stat]{Lemma}
\theoremstyle{definition}

\numberwithin{equation}{section}
\begin{document}%\onehalfspacing
\maketitle
\begin{abstract}%\\
We establish a sharp estimate on the negative moments of the smallest eigenvalue of the Malliavin matrix $\gamma_Z$ of $Z := (u(s, y), u(t, x) - u(s, y))$, where $u$ is the solution to system of $d$ non-linear stochastic heat equations in spatial dimension $k \geq 1$. We also obtain the optimal exponents for the $L^p$-modulus of continuity of the increments of the solution and of its Malliavin derivatives. These lead to optimal lower bounds on hitting probabilities of the process $\{u(t, x): (t, x) \in [0, \infty[ \times \mathbb{R}\}$ in the non-Gaussian case in terms of Newtonian capacity, and improve a result in Dalang, Khoshnevisan and Nualart [\textit{Stoch PDE: Anal Comp} \textbf{1} (2013) 94--151].
\end{abstract}
\vskip.5in %\newpage



\noindent{\it Abbreviated title: Optimal lower bounds on hitting probabilities}

\section{Introduction and main results}

Consider the following system of stochastic partial differential equations:
\begin{align}\label{eq2018-03-01-1}
    \frac{\partial}{\partial t}u_i(t, x) = \frac{1}{2}\Delta_xu_i(t, x) + \sum_{j=1}^d\sigma_{ij}(u(t, x))\dot{F}^j(t, x) + b_i(u(t, x)), & \hbox{}
  \end{align}
for $1 \leq i \leq d$, $t \in [0, T]$ and $x \in \mathbb{R}^k$, where $u: = (u_1, \ldots, u_d)$ with initial conditions $u(0, x) = 0$ for all $x \in \mathbb{R}^k$, and the $\Delta_x$ denotes the Laplacian in the spatial variable $x$. The functions $\sigma_{ij}, \, b_i: \mathbb{R}^d \rightarrow \mathbb{R}$ are globally Lipschitz functions, $i, \, j \in \{1, \ldots, d\}$.  We set $b = (b_i)$ and $\sigma = (\sigma_{ij})$.

The noise $\dot{F} = (\dot{F}_1, \ldots, \dot{F}_d)$ is a spatially homogeneous centered Gaussian generalized random field with covariance of the form
\begin{align}\label{eq2018-03-01-2}
\mbox{E}[\dot{F}^i(t, x)\dot{F}^j(s, y)] = \delta(t - s)\|x - y\|^{-\beta}\delta_{ij},
\end{align}
where $\beta \in \, ]0, 2[$, $\delta(\cdot)$ denotes the Dirac delta function, $\delta_{ij}$ the Kronecker symbol and $\|\cdot\|$ is the Euclidean norm. In particular, the $d$-dimensional driving noise $\dot{F}$ is white in time and with a spatially homogeneous covariance given by the Riesz kernel $f(x) = \|x\|^{-\beta}$.

The solution $u$ of \eqref{eq2018-03-01-1} is known to be a $d$-dimensional random field (see the end of this section where precise definitions and references are given). The potential theory for $u$ has been developed by Dalang, Khoshnevisan and Nualart \cite{DKN13}. Fix $T > 0$ and let $I \times J \subset \, ]0, T] \times \mathbb{R}^k$ be a closed non-trivial rectangle.
%(i.e., $I \subset \, ]0, T]$ is a closed non-trivial interval and $J$ is of the form $[a_1, b_1] \times \cdots \times [a_k, b_k]$ where $a_i, \, b_i \in \mathbb{R}$ and $a_i < b_i$, $i = 1, \ldots, k$).
In the case where the noise is additive, i.e., $\sigma \equiv \mbox{Id}$ and $b \equiv 0$, Dalang, Khoshnevisan and Nualart \cite[Theorem 1.5]{DKN13} prove that for fixed $M > 0$, there exists $c > 0$ depending on $I, J$ and $M$ such that for all compact sets $A \subseteq [-M, M]^d$,
\begin{align} \label{eq2018-08-29-10}
\mbox{P}\{u(I \times J) \cap A \neq \emptyset\} \geq c\, \mbox{Cap}_{d - (\frac{4 + 2k}{2 - \beta})}(A),
\end{align}
where $u(I \times J)$ denotes the range of $I \times J$ under the random map $(t, x) \mapsto u(t, x)$, and $\mbox{Cap}_{\beta}$ denotes the capacity with respect to the Newtonian $\beta$-kernel (we refer to \cite[Section 1]{DKN13} for the definition of capacity).
 If the noise is multiplicative, i.e., $\sigma$ and $b$ are not constants (but are sufficiently regular), then using techniques of Malliavin calculus, Dalang, Khoshnevisan and Nualart \cite[Theorem 1.2(b)]{DKN13} prove that for fixed $M > 0$ and $\eta > 0$, there exists $c > 0$ depending on $I, J, M$ and $\eta$ such that for all compact sets $A \subseteq [-M, M]^d$,
\begin{align} \label{eq2018-08-29-1}
\mbox{P}\{u(I \times J) \cap A \neq \emptyset\} \geq c\, \mbox{Cap}_{d - (\frac{4 + 2k}{2 - \beta}) + \eta}(A).
\end{align}

For systems of linear and/or non-linear stochastic heat equations in spatial dimension $1$ driven by a $d$-dimensional space-time white noise, this type of question was studied in Dalang, Khoshnevisan and Nualart \cite{DKN07} and \cite{DKN09}, in which the lower bounds on hitting probabilities in the Gaussian case and non-Gaussian case are not consistent. This gap has been filled recently by Dalang and Pu \cite{DaP18a}, in which we have obtained the optimal lower bounds on hitting probabilities for systems of non-linear stochastic heat equations in spatial dimension $1$.

The aim of this paper is to remove the $\eta$ in the dimension of capacity in \eqref{eq2018-08-29-1}, so that we obtain the optimal lower bounds on hitting probabilities for systems of non-linear stochastic heat equations in higher spatial dimension.

In \cite{DKN13}, the lower bound on the hitting probability in \eqref{eq2018-08-29-1} follows from the properties of the probability density function of the solution (see \cite[Theorems 1.6 and 1.8]{DKN13}), in particular, the upper bound on the joint probability density function (denoted by $p_{Z}(\cdot, \cdot)$ of the random vector $Z:=(u(s, y), u(t, x) - u(s, y))$. In \cite[Corollary 5.10]{DKN13}, the formula for the density function $p_{Z}(\cdot, \cdot)$) is given in terms of the Malliavin derivative and the Skorohod integral (we refer to Section \ref{section2018-09-21-3} for the elements of Malliavin calculus). From this formula, in order to establish a upper bound on the density function $p_{Z}(\cdot, \cdot)$, the main effort is to analyze the $L^p$-modulus of continuity of the increments of the solution (see \cite[(2.6)]{DKN13}) and of the Malliavin derivative of the increments of the solution (see \cite[Proposition 5.1]{DKN13}), and the negative moments of the smallest eigenvalue of the Malliavin matrix $\gamma_Z$ of $Z$ (see \cite[Proposition 5.6]{DKN13}). We point out that the estimates in \cite[(2.6), Propositions 5.1 and 5.6]{DKN13} are not sharp, and that is why the extra term $\eta$ appears in \eqref{eq2018-08-29-1}.

We first look at the $L^p$-modulus of continuity of the increments of the solution. H\"{o}lder continuity for the solution to stochastic heat equation with spatially correlated noise has been studied by many authors; see, for example,  \cite{HNS13, Lik17, SaS02}. Sanz-Sol\'{e} and Sarr\`{a} \cite{SaS02} use the factorization method to study the H\"{o}lder continuity for the solution to \eqref{eq2018-03-01-1} (with $d = 1$), when the initial condition is bounded and $\rho$-H\"{o}lder continuous  for some $\rho \in \,]0, 1[$, and the spatial covariance of the noise $\dot{F}$ is the Fourier transform of a tempered measure $\mu$ on $\mathbb{R}^k$. In particular, \cite[Theorem 2.1]{SaS02} shows that, if the measure $\mu$ satisfies the condition
\begin{align*}
\int_{\mathbb{R}^k}\frac{\mu{(d\xi)}}{(1 + \|\xi\|^2)^{\eta}} < \infty, \quad \mbox{for some}\quad \eta \in \,]0, 1[,
\end{align*}
then for any $p \geq 2$ and $\gamma \in \, ]0, \rho \wedge(1 - \eta)[$, there exists $C(p, T) >0$ such that for all $(t, x)$, $(s, y) \in [0, T] \times \mathbb{R}^k$,
\begin{align*}
\mbox{E}\left[|u(t, x) - u(s, y)|^p\right] \leq C(p, T)\left(|t - s|^{\gamma/2} + \|x - y\|^{\gamma}\right)^{p};
\end{align*}
see \cite[(10) and (19)]{SaS02}. In the case where $f = \mathscr{F}\mu$ is the Riesz kernel $f(x) = \|x\|^{-\beta}$ and the initial value vanishes, this result of Sanz-Sol\'{e} and Sarr\`{a} becomes: for any $\gamma \in \, ]0, \frac{2 - \beta}{2}[$,
\begin{align} \label{eq2018-11-22-01}
\mbox{E}\left[|u(t, x) - u(s, y)|^p\right] \leq C(p, T)\left(|t - s|^{\gamma/2} + \|x - y\|^{\gamma}\right)^{p}
\end{align}
for all $(t, x)$, $(s, y) \in [0, T] \times \mathbb{R}^k$. Note that the right endpoint $\gamma = \frac{2 - \beta}{2}$ is excluded.

Li \cite{Lik17} has studied the H\"{o}lder continuity for stochastic fractional heat equations without drift in the case where the Gaussian noise is white in time and colored in space with covariance of the form \eqref{eq2018-03-01-2}. Based on some estimates of the fractional heat kernel, \cite[Theorems 1, 2 and 3]{Lik17} obtains spatial and temporal $L^p$-H\"{o}lder continuity of the solution to stochastic fractional heat equation. In these results, the exponent in time is optimal while the exponent in space is not (\cite[Remark 2]{Lik17}).

The first contribution of this paper is the following sharp estimate of the $L^p$-H\"{o}lder continuity for the solution to \eqref{eq2018-03-01-1}, improving \eqref{eq2018-11-22-01}. We have the following.

\begin{theorem}\label{prop2018-09-10-1}
Assume that $\sigma_{ij}$ and $b_i$ are globally Lipschitz continuous. There exists a constant $C_{p, T} > 0$ such that for all $s, t \in [0, T]$, $s \leq t$, $x, y \in \mathbb{R}^k$, $p \geq 2$,
\begin{align}\label{eq2018-09-10-10}
\mathrm{E}\left[\|u(t, x) - u(s, y)\|^p\right] \leq C_{p, T}\left(|t - s|^{(2 - \beta)/2} + \|x - y\|^{2 - \beta}\right)^{p/2}.
\end{align}
\end{theorem}

We also need the $L^p$-H\"{o}lder continuity for the Malliavin derivative of the solution to \eqref{eq2018-03-01-1}. We consider the following hypotheses on the coefficients of the system \eqref{eq2018-03-01-1}:
\begin{enumerate}
  \item [\textbf{P1}] The functions $\sigma_{ij}$ and $b_i$ are infinitely differentiable with bounded partial derivatives of all positive orders, and the $\sigma_{ij}$ are bounded, for $1 \leq i, j \leq d$.
  \item [\textbf{P2}] The matrix $\sigma$ is uniformly elliptic, that is, $\|\sigma(x)\xi\|^2 \geq \rho^2 > 0$ for some $\rho > 0$, for all $x \in \mathbb{R}^d, \|\xi\| = 1$.
\end{enumerate}

Analogous to Theorem \ref{prop2018-09-10-1}, we have the following  sharp estimate of the $L^p$-H\"{o}lder continuity for the Malliavin derivative of the solution to \eqref{eq2018-03-01-1}, which is an improvement of \cite[Proposition 5.1]{DKN13}.
\begin{theorem}\label{prop2018-09-11-1}
Assume \textbf{P1}. Then for any $T > 0$ and $p \geq 2$, there exists a constant $C := C_{p, T} > 0$ such that for any $0 \leq s\leq t\leq T$, $x, y \in \mathbb{R}^k$, $m \geq 1$  and $i \in \{1, \ldots, d\}$,
\begin{align}\label{eq2018-09-11-1}
\|D^m(u_i(t, x) - u_i(s, y))\|_{L^p(\Omega; (\mathscr{H}_T^d)^{\otimes m})} \leq C\big(|t - s|^{\frac{2 - \beta}{4}} + \|x - y\|^{\frac{2 - \beta}{2}}\big).
\end{align}
\end{theorem}

Theorems \ref{prop2018-09-10-1} and \ref{prop2018-09-11-1} are proved in Section \ref{section2018-09-20-1}.
Turning to the negative moments of the smallest eigenvalue of the Malliavin matrix $\gamma_Z$, we have the following sharp estimate, which is an improvement of \cite[Proposition 5.6]{DKN13}.

\begin{theorem}\label{prop2017-09-19-4000}
 Assume \textbf{P1} and \textbf{P2}. Fix $T > 0$ and let $I \times J \subset \, ]0, T] \times \mathbb{R}^k$ be a closed non-trivial rectangle. There exists $C > 0$ depending on $T$ such that for all $s, t \in I$, $0 \leq t - s < 1$, $x, y \in J$, $(s, y) \neq (t, x)$, and $p > 1$,
      \begin{align}\label{eq2017-09-19-5}
      \mathrm{E}\Bigg[\bigg(\inf_{\scriptsize\begin{array}{c} \xi = (\lambda, \mu) \in \mathbb{R}^{2d}:\\ \|\lambda\|^2 + \|\mu\|^2 = 1 \end{array}}\xi^T\gamma_Z\xi\bigg)^{-2dp}\Bigg] &\leq C(|t - s|^{\frac{2 - \beta}{2}} + \|x - y\|^{2 - \beta})^{-2dp}.
            \end{align}
\end{theorem}

Theorem \ref{prop2017-09-19-4000} is proved in Section \ref{section2018-09-21-1}.
Using Theorems \ref{prop2018-09-10-1}, \ref{prop2018-09-11-1} and \ref{prop2017-09-19-4000} and some results of \cite{DKN13}, we establish a sharp upper bound on the joint probability density function (denoted by $p_{s, y; t, x}(\cdot, \cdot)$) of the random vector $(u(s, y), u(t, x))$  and the optimal lower bounds on hitting probabilities of the solution to \eqref{eq2018-03-01-1}.
\begin{theorem}\label{theorem2017-08-17-1}
Assume \textbf{P1} and \textbf{P2}. Fix $T > 0$ and let  $I \times J \subset \, ]0, T] \times \mathbb{R}^k$ be a closed non-trivial rectangle.
There exists $c > 0$ such that for all $s, t \in I$, $x, y \in J$ with $(s, y) \neq (t, x)$, $z_1, z_2 \in \mathbb{R}^d$ and $p \geq 1$,
\begin{align} \label{eq2017-11-20-1}
p_{s, y; t, x}(z_1, z_2) &\leq c(|t - s|^{\frac{2 - \beta}{2}} + \|x - y\|^{2 - \beta})^{-\frac{d}{2}}\left[\frac{(|t - s|^{\frac{2 - \beta}{2}} + \|x - y\|^{2 - \beta})^2}{\|z_1 - z_2\|^2} \wedge 1\right]^{p/(2d)}.
\end{align}
\end{theorem}
\begin{theorem}\label{theorem2018-09-21-1}
Assume \textbf{P1} and \textbf{P2}. Fix $T > 0$ and $M > 0$. Let $I \times J \subset \, ]0, T] \times \mathbb{R}^k$ be a closed non-trivial rectangle. There exists $c > 0$ depending on $I, J$ and $M$ such that for all compact sets $A \subseteq [-M, M]^d$,
\begin{align} \label{th2018-09-21-2}
\mathrm{P}\{u(I \times J) \cap A \neq \emptyset\} \geq c\, \rm{Cap}_{d - (\frac{4 + 2k}{2 - \beta})}(A).
\end{align}
\end{theorem}

Theorem \ref{theorem2017-08-17-1} is an improvement of \cite[Theorem 1.6(b)]{DKN13} and Theorem \ref{theorem2018-09-21-1} is an improvement of \cite[Theorem 1.2(b)]{DKN13}, and they are proved in Section \ref{section2018-09-21-4}. The main ingredients which allow for these improvements are the sharp $L^p$-H\"{o}lder continuity estimates of Theorems \ref{prop2018-09-10-1} and \ref{prop2018-09-11-1} and a better estimate on the Malliavin derivative of $u(t, x)$ given in Lemma \ref{lemma2018-09-12-1} below.

We conclude this section by giving a rigorous formulation of \eqref{eq2018-03-01-1}, following Walsh \cite{Wal86}.
We first define precisely the driving noise that appears in \eqref{eq2018-03-01-1}. Let "$\cdot$" denote the temporal variable and "$\ast$" the spatial variable. Let $\mathscr{D}(\mathbb{R}^{k + 1})$ be the space of $C^{\infty}$ test-functions with compact support. Then $F = \{F(\phi) = (F^1(\phi), \ldots, F^d(\phi)), \phi \in \mathscr{D}(\mathbb{R}^{k + 1})\}$ is an $L^2(\Omega, \mathscr{F}, \mbox{P})^d$-valued mean zero Gaussian process with covariance
\begin{align*}
\mbox{E}\left[F^i(\phi)F^j(\psi)\right] = \delta_{ij}\int_{\mathbb{R}_+}dr \int_{\mathbb{R}^k}dy\int_{\mathbb{R}^k}dz \, \phi(r, y)\|y - z\|^{-\beta}\psi(r, z).
\end{align*}
Using elementary properties of the Fourier transform (see Dalang \cite{Dal99}), this covariance can also be written as
\begin{align*}
\mbox{E}\left[F^i(\phi)F^j(\psi)\right] = \delta_{ij}\, c_{k, \beta}\int_{\mathbb{R}_+}dr \int_{\mathbb{R}^k}d\xi \, \|\xi\|^{\beta - k} \mathscr{F}\phi(r, *)(\xi)\overline{\mathscr{F}\psi(r, *)(\xi)},
\end{align*}
where $c_{k, \beta}$ is a constant and $\mathscr{F}f(\xi)$ is the Fourier transform of $f$, that is,
\begin{align*}
\mathscr{F}f(\xi) = \int_{\mathbb{R}^k}e^{- i \xi\cdot x}f(x)dx.
\end{align*}

Following Walsh \cite{Wal86}, a rigorous formulation of \eqref{eq2018-03-01-1} through the notion of {\em mild solution} is as follows. Let $M = (M^1, \ldots, M^d), M^i = \{M_t^i(A), t \geq 0, A \in \mathscr{B}_b(\mathbb{R}^k)\}$ be the $d$-dimensional worthy martingale measure obtained as an extension of the process $\dot{F}$ as in Dalang and Frangos \cite{DaF98}. Then a {\em mild solution} of \eqref{eq2018-03-01-1} is a jointly measurable $\mathbb{R}^d$-valued process $u = \{u(t, x), t \geq 0, x \in \mathbb{R}^k\}$, adapted to the natural filtration generated by $M$, such that
\begin{align}\label{eq2018-03-01-3}
u_i(t, x) & = \int_0^t\int_{\mathbb{R}^k}S(t - s, x - y)\sum_{j=1}^d\sigma_{ij}(u(s, y))M^j(ds, dy) \nonumber \\
& \quad + \int_0^tds\int_{\mathbb{R}^k}dy \, S(t - s, x - y)b_i(u(s, y)), \qquad i \in \{1, \ldots, d\},
\end{align}
where $S(t, x)$ is the fundamental solution of the deterministic heat equation in $\mathbb{R}^k$, that is,
\begin{align*}
S(t, x) = (2\pi t)^{-k/2}\exp\left(-\frac{\|x\|^2}{2t}\right),
\end{align*}
and the stochastic integral is interpreted in the sense of \cite{Wal86}.

Using the results of Dalang \cite{Dal99}, existence and uniqueness of the solution of \eqref{eq2018-03-01-1} holds,
as discussed in \cite[Section 2]{DKN13}, under the condition
\begin{align}\label{eq2018-09-11-3}
0 < \beta < (2 \wedge k),
\end{align}
and in this case, there exists a unique $L^2$-continuous solution of \eqref{eq2018-03-01-3} satisfying
\begin{align}\label{eq2018-09-10-1}
\sup_{(t, x) \in [0, T] \times \mathbb{R}^k}\mbox{E}\left[|u_i(t, x)|^p\right] < \infty, \quad i \in \{1, \ldots, d\},
\end{align}
for any $T > 0$ and $p \geq 1$.

\section{Elements of Malliavin calculus}\label{section2018-09-21-3}

In this section, we introduce, following Nualart \cite{Nua06} (see also \cite{San05}), some elements of Malliavin calculus.
Let $\mathscr{S}(\mathbb{R}^k)$  be the Schwartz space of $C^{\infty}$ functions on $\mathbb{R}^k$ with rapid decrease. Let $\mathscr{H}$ denote the completion of $\mathscr{S}(\mathbb{R}^k)$ endowed with their inner product
\begin{align*}
\langle \phi, \psi\rangle_{\mathscr{H}} &= \int_{\mathbb{R}^k}dx\int_{\mathbb{R}^k}dy \, \phi(x)\|x - y\|^{-\beta}\psi(y) \\
& = c_{k, \beta}\int_{\mathbb{R}^{k}}d\xi \, \|\xi\|^{\beta - k} \mathscr{F}\phi(\xi)\overline{\mathscr{F}\psi(\xi)},
\end{align*}
$\phi, \, \psi \in \mathscr{S}(\mathbb{R}^k)$.

For $h = (h_1, \ldots, h_d) \in \mathscr{H}^d$ and $\tilde{h} = (\tilde{h}_1, \ldots, \tilde{h}_d) \in \mathscr{H}^d$, we set $\langle h, \tilde{h}\rangle_{\mathscr{H}^d} = \sum_{i = 1}^d\langle h_i, \tilde{h}_i\rangle_{\mathscr{H}}$.
Let $T > 0$ be fixed. We set $\mathscr{H}_T^d = L^2([0, T]; \mathscr{H}^d)$ and for $0 \leq s \leq t \leq T$, we will write $\mathscr{H}_{s, t}^d = L^2([s, t], \mathscr{H}^d)$.

The centered Gaussian noise $F$ can be used to construct an isonormal Gaussian process $\{W(h), h \in \mathscr{H}_T^d\}$ as follows. Let $\{e_j, j \geq 0\} \subset \mathscr{S}(\mathbb{R}^k)$ be a complete orthonormal system of the Hilbert space $\mathscr{H}$. Then for any $t \in [0, T]$, $i \in \{1, \ldots, d\}$ and $j \geq 0$, set
\begin{align*}
W_j^i(t) = \int_0^t\int_{\mathbb{R}^k}e_j(x)\cdot F^i(ds, dx),
\end{align*}
so that $(W_j^i, j \geq 1)$ is a sequence of independent standard real-valued Brownian motions such that, for any $\phi \in \mathscr{D}([0, T] \times \mathbb{R}^k)$, \begin{align*}
F^i(\phi) = \sum_{j = 0}^{\infty} \int_0^T \left\langle \phi(s, *), e_j(*)\right\rangle_{\mathscr{H}}dW^i_j(s),
\end{align*}
where the series converges in $L^2(\Omega, \mathscr{F}, \mbox{P})$. For $h^i \in \mathscr{H}_T$, we set
\begin{align*}
W^i(h) = \sum_{j = 0}^{\infty} \int_0^T \left\langle h^i(s, *), e_j(*)\right\rangle_{\mathscr{H}}dW^i_j(s),
\end{align*}
where, again, the series converges in $L^2(\Omega, \mathscr{F}, \mbox{P})$. In particular, for $\phi \in \mathscr{D}([0, T] \times \mathbb{R}^k)$, $F^i(\phi) = W^i(\phi)$. Finally, for $h = (h_1, \ldots, h_d) \in \mathscr{H}_T^d$, we set
\begin{align*}
W(h) = \sum_{i = 1}^dW^i(h^i).
\end{align*}

With this isonormal Gaussian process, we can use the framework of Malliavin calculus. Let $\mathscr{S}$ denote the class of smooth random variables of the form
\begin{align*}
G = g(W(h_1), \ldots, W(h_n)),
\end{align*}
where $n \geq 1$, $g \in \mathscr{C}_p^{\infty}(\mathbb{R}^n)$, the set of real-valued functions $g$ such that $g$ and all its partial derivatives have at most polynomial growth and $h_i \in \mathscr{H}_T^d$. Given $G \in \mathscr{S}$, its derivative $(D_{r}G = (D^{(1)}_{r}G, \ldots, D^{(d)}_{r}G), \, r \in [0, T])$ is an $\mathscr{H}_T^d$-valued random vector defined by
\begin{align*}
D_{r}G = \sum_{i = 1}^n \partial_i g (W(h_1), \ldots, W(h_n))h_i(r).
\end{align*}
For $\phi \in \mathscr{H}^d$ and $r \in [0, T]$, we write $D_{r, \phi}G = \langle D_rG, \phi(*)\rangle_{\mathscr{H}^d}$.
More generally,  the derivative $D^mG = (D^m_{(r_1, \ldots, r_m)}, (r_1, \ldots, r_m) \in [0, T]^m)$ of order $m \geq 1$ of $G$ is the $(\mathscr{H}_T^d)^{\otimes m}$-valued random vector defined by
 \begin{align*}
D_{(r_1, \ldots, r_m)}^mG = \sum_{i_1, \ldots, i_m = 1}^n \frac{\partial}{\partial x_{i_1}}\cdots \frac{\partial}{\partial x_{i_m}}\, g(W(h_1), \ldots, W(h_n))h_{i_1}(r_1)\otimes \cdots \otimes h_{i_m}(r_m),
\end{align*}
where the notation $\otimes$ denotes the tensor product of functions.

For $p, m \geq 1$, the space $\mathbb{D}^{m, p}$ is the closure of $\mathscr{S}$ with respect to the seminorm $\|\cdot\|_{m, p}$ defined by
\begin{align*}
\|G\|_{m, p}^p = \mbox{E}[|G|^p] + \sum_{j = 1}^m\mbox{E}\left[\|D^jG\|_{(\mathscr{H}_T^d)^{\otimes j}}^p\right].
\end{align*}
We set $\mathbb{D}^{\infty} = \cap_{p \geq 1}\cap_{m \geq 1}\mathbb{D}^{m, p}$.

The derivative operator $D$ on $L^2(\Omega)$ has an adjoint, termed the Skorohod integral and denoted by $\delta$, which is an unbounded and closed operator on $L^2(\Omega, \mathscr{H}_T^d)$; see \cite[Section 1.3]{Nua06}. Its domain, denoted by  $\mbox{Dom}\ \delta$,  is the set of elements $u \in L^2(\Omega, \mathscr{H}_T^d)$ such that there exists a constant $c$ such that $|\mbox{E}[\langle DG, u \rangle_{\mathscr{H}_T^d}]| \leq c \|G\|_{0, 2}$, for any $G \in \mathbb{D}^{1, 2}$. If $u \in \mbox{Dom}\ \delta$, then $\delta(u)$ is the element of $L^2(\Omega)$ characterized by the following duality relation:
\begin{align*}
\mbox{E}[G\, \delta(u)] = \mbox{E}\left[\langle DG, u\rangle_{\mathscr{H}_T^d}\right], \quad \mbox{for all }\ G \in \mathbb{D}^{1, 2}.
\end{align*}

Recall from \cite[Section 3]{DKN13} that for $r \in [0, t]$ and $i, l \in \{1, \ldots, d\}$, the derivative of $u_i(t, x)$ satisfies the system of equations
\begin{align}\label{eq2018-03-01-4}
D_r^{(l)}(u_i(t, x)) & = \sigma_{il}(u(r, *))S(t - r, x - *) + a_i(l, r, t, x),
\end{align}
where
\begin{align}\label{eq2018-09-12-1}
 a_i(l, r, t, x) &=  \int_r^t\int_{\mathbb{R}^k} S(t - \theta, x - \eta)\sum_{j = 1}^dD_r^{(l)}(\sigma_{ij}(u(\theta, \eta)))M^j(d\theta, d\eta) \nonumber \\
& \quad + \int_r^td\theta \int_{\mathbb{R}^k}d\eta \, S(t - \theta, x - \eta)D_r^{(l)}(b_i(u(\theta, \eta))),
\end{align}
and $D_r^{(l)}(u_i(t, x)) = 0$ if $r > t$. Moreover, by \cite[Proposition 6.1]{NuQ07}, for any $p > 1$, $m \geq 1$ and $i \in \{1, \ldots, d\}$, the order $m$ derivative satisfies
\begin{align}\label{eq2018-09-10-2}
\sup_{(t, x) \in [0, T] \times \mathbb{R}^k}\mbox{E}\left[\|D^m(u_i(t, x))\|^p_{(\mathscr{H}_T^d)^{\otimes m}}\right] < \infty,
\end{align}
and $D^m$ also satisfies the system of stochastic partial differential equations given in \cite[(6.29)]{NuQ07} and obtained by iterating the calculation that leads to \eqref{eq2018-03-01-4}. In particular, $u(t, x) \in (\mathbb{D}^{\infty})^d$, for all $(t, x) \in [0, T] \times \mathbb{R}^k$.

\section{Proof of Theorems \ref{theorem2017-08-17-1} and \ref{theorem2018-09-21-1} (assuming Theorems \ref{prop2018-09-10-1}--\ref{prop2017-09-19-4000})} \label{section2018-09-21-4}

Recall that the Malliavin matrix  $\gamma_Z$ of $Z = (u(s, y), u(t, x) - u(s, y))$ is a symmetric $2d \times 2d$ random matrix with four $d \times d$ blocs of the form
\begin{align*}
\gamma_Z = \left(
  \begin{array}{ccc}
    \gamma_Z^{(1)} & \vdots  & \gamma_Z^{(2)} \\
    \cdots  & \vdots  & \cdots  \\
    \gamma_Z^{(3)} &  \vdots & \gamma_Z^{(4)} \\
  \end{array}
\right)
\end{align*}
where
\begin{align*}
\gamma_Z^{(1)} &= \left(\left\langle D(u_i(s, y)), D(u_j(s, y))\right\rangle_{\mathscr{H}_T^d}\right)_{i, j = 1, \dots, d},\\
\gamma_Z^{(2)} &=   \left(\left\langle D(u_i(s, y)), D(u_j(t, x)- u_j(s, y))\right\rangle_{\mathscr{H}_T^d}\right)_{i, j = 1, \dots, d},\\
\gamma_Z^{(3)} &=  \left(\left\langle D(u_i(t, x)- u_i(s, y)), D(u_j(s, y))\right\rangle_{\mathscr{H}_T^d}\right)_{i, j = 1, \dots, d},\\
\gamma_Z^{(4)} &=  \left(\left\langle D(u_i(t, x) - u_i(s, y)), D(u_j(t, x)- u_j(s, y))\right\rangle_{\mathscr{H}_T^d}\right)_{i, j = 1, \dots, d}.
\end{align*}
We let ($\mathbf{1}$) denote the couples of $\{1, \ldots, d\} \times \{1, \ldots, d\}$, ($\mathbf{2}$) denote the couples of $\{1, \ldots, d\} \times \{d + 1, \ldots, 2d\}$, ($\mathbf{3}$) denote the couples of $\{d + 1, \ldots, 2d\} \times \{1, \ldots, d\}$ and ($\mathbf{4}$) denote the couples of $\{d + 1, \ldots, 2d\} \times \{d + 1, \ldots, 2d\}$.

We first state two results which follow exactly along the same lines as \cite[Propositions 5.3 and 5.4]{DKN13}, using \eqref{eq2018-09-10-2} and our Theorem \ref{prop2018-09-11-1} instead of their \cite[(3.2) and Proposition 5.1]{DKN13}. Their proofs are omitted.

\begin{prop}\label{prop2017-09-19-1}
Fix $T > 0$ and let  $I \times J \subset \, ]0, T] \times \mathbb{R}^k$ be a closed non-trivial rectangle. Let $A_Z$ denote the cofactor matrix of $\gamma_Z$. Assuming \textbf{P1}, for any $(s, y), (t, x) \in I \times J$, $(s, y) \neq (t, x)$, $p > 1$,
\begin{align*}
\mathrm{E}\left[|(A_Z)_{m, l}|^p\right]^{1/p}
 \leq
\left\{\begin{array}{ll}
    c_{p, T} (|t - s|^{\frac{2 - \beta}{2}} + \|x - y\|^{2 - \beta})^d & \hbox{if $(m, l) \in (\mathbf{1})$,} \\
    c_{p, T} (|t - s|^{\frac{2 - \beta}{2}} + \|x - y\|^{2 - \beta})^{d - \frac{1}{2}}   & \hbox{if $(m, l) \in (\mathbf{2})$ or $(\mathbf{3})$,} \\
    c_{p, T} (|t - s|^{\frac{2 - \beta}{2}} + \|x - y\|^{2 - \beta})^{d - 1} & \hbox{if $(m, l) \in (\mathbf{4})$.}
  \end{array}
\right.
\end{align*}
\end{prop}
\begin{prop}\label{prop2017-09-19-2}
Fix $T > 0$ and let  $I \times J \subset \, ]0, T] \times \mathbb{R}^k$ be a closed non-trivial rectangle. Assuming \textbf{P1}, for any $(s, y), (t, x) \in I \times J$, $(s, y) \neq (t, x)$, $p > 1$,
\begin{align*}
\mathrm{E}\left[\|D^k(\gamma_Z)_{m, l}\|_{\mathscr{H}^{\otimes k}}^p\right]^{1/p}
 \leq
\left\{\begin{array}{ll}
    c_{k, p, T} & \hbox{if $(m, l) \in (\mathbf{1})$,} \\
    c_{k, p, T} (|t - s|^{\frac{2 - \beta}{2}} + \|x - y\|^{2 - \beta})^{\frac{1}{2}}   & \hbox{if $(m, l) \in (\mathbf{2})$ or $(\mathbf{3})$,} \\
    c_{k, p, T} (|t - s|^{\frac{2 - \beta}{2}} + \|x - y\|^{2 - \beta}) & \hbox{if $(m, l) \in (\mathbf{4})$.}
  \end{array}
\right.
\end{align*}
\end{prop}

The next result is an improvement of \cite[Proposition 5.5]{DKN13}.
\begin{prop}\label{prop2017-09-19-3}
Fix $T > 0$ and let  $I \times J \subset \, ]0, T] \times \mathbb{R}^k$ be a closed non-trivial rectangle.  Assume \textbf{P1} and \textbf{P2}. There exists $C$ depending on $T$ such that for any $(s, y), (t, x) \in I \times J$, $(s, y) \neq (t, x)$, $p > 1$,
\begin{align}\label{eq2017-09-19-1}
\mbox{E}\left[(\det \gamma_Z)^{-p}\right]^{1/p} \leq C (|t - s|^{\frac{2 - \beta}{2}} + \|x - y\|^{2 - \beta})^{-d}.
\end{align}
\end{prop}
\begin{proof}
Similar to the proof of \cite[Proposition 5.5]{DKN13} (see also \cite[Proposition 6.6]{DKN09}), this is a consequence of \cite[(5.11)]{DKN13}, our Theorem \ref{prop2017-09-19-4000} and \cite[Proposition 5.7]{DKN13}.
\end{proof}

From Propositions \ref{prop2017-09-19-1}--\ref{prop2017-09-19-3}, we obtain the following result, which improves \cite[Theorem 5.8]{DKN13}. The proof is similar to that of \cite[Theorem 5.8]{DKN13} (but using our Proposition \ref{prop2017-09-19-3} instead of \cite[Proposition 5.5]{DKN13}) and hence is omitted.
\begin{prop}\label{th2017-09-22-1}
Fix $T > 0$  and let  $I \times J \subset \, ]0, T] \times \mathbb{R}^k$ be a closed non-trivial rectangle.  Assume \textbf{P1} and \textbf{P2}.  For any $(s, y), (t, x) \in I \times J$,  $(s, y) \neq (t, x)$, $k \geq 0$ and  $p > 1$,
\begin{align}\label{eq2017-09-22-1}
\mbox{E}\left[\|(\gamma_Z)^{-1}_{m, l}\|_{k, p}\right]
 \leq
\left\{\begin{array}{ll}
    c_{k, p, T} & \hbox{if $(m, l) \in (\mathbf{1})$,} \\
    c_{k, p, T} (|t - s|^{\frac{2 - \beta}{2}} + \|x - y\|^{2 - \beta})^{-\frac{1}{2}}   & \hbox{if $(m, l) \in (\mathbf{2})$ or $(\mathbf{3})$,} \\
    c_{k, p, T} (|t - s|^{\frac{2 - \beta}{2}} + \|x - y\|^{2 - \beta})^{-1} & \hbox{if $(m, l) \in (\mathbf{4})$.}
  \end{array}
\right.
\end{align}
\end{prop}

We are now ready to prove Theorems \ref{theorem2017-08-17-1} and \ref{theorem2018-09-21-1}.

\begin{proof}[Proof of Theorem \ref{theorem2017-08-17-1}]
We recall from the proof of \cite[Theorem 1.6(b)]{DKN13} that
\begin{align*}
p_{s, y; t, x}(z_1, z_2) = p_Z(z_1, z_2 - z_1), \quad \mbox{for all} \quad z_1, \, z_2 \in \mathbb{R}^d,
\end{align*}
and
\begin{align}\label{eq2017-09-21-5}
p_Z(z_1, z_2 - z_1) &\leq \prod_{i = 1}^d\left(\mbox{P}\left\{|u_i(t, x) - u_i(s, y)| > |z_1^i - z_2^i|\right\}\right)^{\frac{1}{2d}}\nonumber \\
& \quad \times \|H_{(1, \ldots, 2d)}(Z, 1)\|_{0, 2},
\end{align}
where the random variable $H_{(1, \ldots, 2d)}(Z, 1)$ is given by the formula in \cite[Corollary 5.10]{DKN13}.
Using Chebyshev's inequality and Theorem \ref{prop2018-09-10-1}, we see that
\begin{align*}
\prod_{i = 1}^d\left(\mbox{P}\left\{|u_i(t, x) - u_i(s, y)| > |z_1^i - z_2^i|\right\}\right)^{\frac{1}{2d}} \leq c \left[\frac{|t - s|^{\frac{2 - \beta}{2}} + \|x - y\|^{2 - \beta}}{\|z_1 - z_2\|^2} \wedge 1\right]^{p/(2d)}.
\end{align*}
It remains to prove that
\begin{align}\label{eq2018-09-21-1}
& \|H_{(1, \ldots, 2d)}(Z, 1)\|_{0, 2} \leq c_T(|t - s|^{\frac{2 - \beta}{2}} + \|x - y\|^{2 - \beta})^{-d/2}.
\end{align}
The proof of \eqref{eq2018-09-21-1} is similar to that of \cite[Proposition 5.11]{DKN13} by using the continuity of the Skorohod integral $\delta$ (see \cite[Proposition 3.2.1]{Nua06} and \cite[(1.11) and p.131]{Nua98}) and H\"{o}lder's inequality for Malliavin norms (see \cite[Proposition 1.10, p.50]{Wat84}). Comparing with the estimate in \cite[Proposition 5.11]{DKN13}, we are able to remove the extra exponent $\eta$ because of the correct estimate on the inverse of the matrix $\gamma_Z$ in Proposition \ref{th2017-09-22-1}.
\end{proof}

\begin{proof}[Proof of  Theorems \ref{theorem2018-09-21-1}]
The proof is similar to that of \cite[Theorem 1.2(b)]{DKN13}. We remark that the estimate in \cite[Lemma 2.3]{DKN13} remains valid for $\tilde{\gamma} = \gamma = 2 - \beta$. Then by using Theorem \ref{theorem2017-08-17-1}, we follow along the same lines as in the proof of \cite[Theorem 1.2(b)]{DKN13} with $d + \eta$ there replaced by $d$, to obtain the optimal lower bounds on hitting probabilities in terms of capacity.
\end{proof}

\section{Proof of Theorems \ref{prop2018-09-10-1} and \ref{prop2018-09-11-1}} \label{section2018-09-20-1}

In this section, we establish the $L^p$-H\"{o}lder continuity of the solution and its Malliavin derivative. First, we recall some estimates on the Green kernel $S(t, x)$.
\begin{lemma}\label{lemma2018-09-10-1}
There exist some $M_0, \, m_0 >0$ such that for all $t > 0$, $x, \, y \in \mathbb{R}^k$,
\begin{align}\label{eq2018-09-10-3}
|S(t, x) - S(t, y)| \leq M_0 \|x - y\|t^{-\frac{k + 1}{2}}\big(e^{-\frac{\|x\|^2}{m_0t}} + e^{-\frac{\|y\|^2}{m_0t}}\big).
\end{align}
\end{lemma}
\begin{proof}
This is a consequence of the mean-value theorem.
\end{proof}

\begin{lemma}[{{\cite[Lemma 6.4]{CJKS13}}}]\label{lemma2018-09-10-2}
There exist some $N_0>0$ such that for all $t > 0$, $x, \, y \in \mathbb{R}^k$,
\begin{align}\label{eq2018-09-10-4}
\int_{\mathbb{R}^k}|S(t, x + z) - S(t, y + z)|dz \leq N_0 \big(\frac{\|x - y\|}{t^{1/2}} \wedge 1\big).
\end{align}
\end{lemma}

\begin{lemma}\label{lemma2018-09-10-3}
There exists a constant $C_0>0$ such that for all $t > 0$, $x, \, y \in \mathbb{R}^k$,
\begin{align}\label{eq2018-09-10-5}
\int_{\mathbb{R}^k}\|z\|^{-\beta}|S(t, x + z) - S(t, y + z)|dz \leq C_0 t^{-\beta/2} \big(\frac{\|x - y\|}{t^{1/2}} \wedge 1\big).
\end{align}
\end{lemma}
\begin{proof}
First, by Lemma \ref{lemma2018-09-10-1},
\begin{align}\label{eq2018-09-10-6}
&\int_{\mathbb{R}^k}\|z\|^{-\beta}|S(t, x + z) - S(t, y + z)|dz \nonumber \\
&\quad \leq M_0 \|x - y\|\, t^{-\frac{k + 1}{2}}\int_{\mathbb{R}^k}\|z\|^{-\beta}\big(e^{-\frac{\|x + z\|^2}{m_0t}} + e^{-\frac{\|y + z\|^2}{m_0t}}\big)dz \nonumber \\
& \quad \leq  2 M_0 \|x - y\|\, t^{-\frac{k + 1}{2}} \sup_{x \in \mathbb{R}^k}\int_{\mathbb{R}^k}\|z\|^{-\beta}e^{-\frac{\|x + z\|^2}{m_0t}}dz \nonumber \\
& \quad = 2 M_0 \|x - y\|\, t^{-\frac{k + 1}{2}}\int_{\mathbb{R}^k}\|z\|^{-\beta}e^{-\frac{\|z\|^2}{m_0t}}dz \nonumber \\
& \quad = C \|x - y\|\, t^{-\frac{\beta + 1}{2}},
\end{align}
where the first equality holds since the function $x \mapsto \int_{\mathbb{R}^k}\|z\|^{-\beta}e^{-\|x + z\|^2/(m_0t)}dz$ is a nonnegative definite function (its Fourier transform is a nonnegative function), which is therefore maximized at $x = 0$.

On the other hand, using the same arguments as above,
\begin{align}\label{eq2018-09-10-7}
&\int_{\mathbb{R}^k}\|z\|^{-\beta}|S(t, x + z) - S(t, y + z)|dz \nonumber \\
& \quad \leq \int_{\mathbb{R}^k}\|z\|^{-\beta}|S(t, x + z) + S(t, y + z)|dz \nonumber \\
& \quad \leq 2\sup_{x \in \mathbb{R}^k}\int_{\mathbb{R}^k}\|z\|^{-\beta}S(t, x + z)dz \nonumber \\
& \quad = 2\int_{\mathbb{R}^k}\|z\|^{-\beta}S(t, z)dz  = C \, t^{-\frac{\beta}{2}}.
\end{align}
Therefore, \eqref{eq2018-09-10-7} and \eqref{eq2018-09-10-6} imply \eqref{eq2018-09-10-5}.
\end{proof}

\begin{lemma}\label{lemma2018-09-10-4}
There exists a constant $C>0$ such that for all $s \geq 0$, $x, \, y \in \mathbb{R}^k$,
\begin{align}\label{eq2018-09-10-8}
& \int_0^sdr\int_{\mathbb{R}^k}dz\int_{\mathbb{R}^k}dv \, \|z - v\|^{-\beta}|S(r, x - z) - S(r, y - z)|\, |S(r, x - v) - S(r, y - v)| \nonumber \\
& \quad \leq C \|x - y\|^{2 - \beta}.
\end{align}
\end{lemma}
\begin{proof}
By the change of variable $[z - v = v']$, the integral in \eqref{eq2018-09-10-8} is equal to
\begin{align*}
& \int_0^sdr\int_{\mathbb{R}^k}dz \, |S(r, x - z) - S(r, y - z)| \nonumber \\
& \quad \quad \times \int_{\mathbb{R}^k}dv' \, \|v'\|^{-\beta} |S(r, x - z + v') - S(r, y -z + v')|.
\end{align*}
Applying Lemma \ref{lemma2018-09-10-3} first and then Lemma \ref{lemma2018-09-10-2}, this is bounded above by
\begin{align*}
& C \int_0^sdr\, r^{-\beta/2}\Big(\frac{\|x - y\|}{r^{1/2}} \wedge 1\Big)\int_{\mathbb{R}^k}dz \, |S(r, x - z) - S(r, y - z)| \nonumber \\ \nonumber \\
& \quad \leq  C \int_0^sdr\, r^{-\beta/2}\Big(\frac{\|x - y\|^{2}}{r} \wedge 1\Big)\leq  C \int_0^{\infty}dr\, r^{-\beta/2}\Big(\frac{\|x - y\|^{2}}{r} \wedge 1\Big) \nonumber \\
& \quad = C\int_0^{\|x - y\|^{2}}r^{-\beta/2}dr + C \|x - y\|^{2} \int_{\|x - y\|^{2}}^{\infty}r^{-\beta/2 - 1}dr \nonumber \\
& \quad = C \|x - y\|^{2 - \beta} + C \|x - y\|^{2 + 2\times (1 - 1 - \beta/2)} = 2C \|x - y\|^{2 - \beta}. \qedhere
\end{align*}
\end{proof}

\begin{lemma}\label{lemma2018-09-10-5}
There exists a constant $C>0$ such that for all $t \geq 0$, $\delta \geq 0$ and $x \in \mathbb{R}^k$,
\begin{align}\label{eq2018-09-10-9}
& \int_0^tdr\int_{\mathbb{R}^k}dz\int_{\mathbb{R}^k}dv \, \|z - v\|^{-\beta}|S(r + \delta, x - z) - S(r, x - z)|\, |S(r + \delta, x - v) - S(r, x - v)| \nonumber \\
& \quad \leq C \delta^{(2 - \beta)/2}.
\end{align}
\end{lemma}
\begin{proof}
The proof follows the same lines as the  estimate of $I_1$ in the proof of \cite[Theorem 2]{Lik17} applied to the case of the ordinary stochastic heat equation. We point out that,  in the case of the ordinary heat kernel, the argument for the estimate in \cite[(2.26)]{Lik17} does not apply since the lower bound in \cite[(1.7)]{Lik17} does not apply to the heat kernel. However, in the case of the ordinary heat kernel, the statement of \cite[(2.26)]{Lik17} (together with \cite[(2.27) and (2.28)]{Lik17}) remains valid by the following calculation:
\begin{align*}
\int_{\mathbb{R}^k}\|x\|^{- \beta}S(r, x)dx
&= \int_{\mathbb{R}^k}r^{-\beta/2}\|z\|^{- \beta} S(1, z)dz =C\, r^{-\beta/2}.
\end{align*}

Moreover, we can replace the inequality in \cite[(2.33)]{Lik17} by the following: for $r \in \,]0, 1]$, there exists for $C > 1$ such that for all $\mu >0$,
\begin{align}\label{eq2018-11-16-1}
0 < \log(1 + \mu) \leq C\, \mu^r.
\end{align}

We apply \eqref{eq2018-11-16-1} to conclude that
\begin{align}\label{eq2018-11-16-2}
0 < \log(1 + \delta/t) \leq C\, (\delta/t)^{(2 - \beta)/2}, \quad \mbox{for all} \,\,\, \delta \geq 0, \, t > 0.
\end{align}
Using \eqref{eq2018-11-16-2} instead of \cite[(2.35)]{Lik17}, the remaining calculation is the same as that in \cite[(2.29)--(2.32)]{Lik17}.
\end{proof}

Based on the above estimates on the Green kernel, we now prove Theorems \ref{prop2018-09-10-1} and \ref{prop2018-09-11-1}.

\begin{proof}[Proof of Theorem \ref{prop2018-09-10-1}]
By \eqref{eq2018-09-10-1}, it suffices to prove \eqref{eq2018-09-10-10} when $t - s$ and $\|x - y\|$ are small. Without loss of generality, we assume that $t - s \leq 1/2$ and $\|x - y\| \leq 1/2$.
Denote
$$I_{ij}(t, x) = \int_0^t\int_{\mathbb{R}^k}S(t - \theta, x - \eta)\sigma_{ij}(u(\theta, \eta))M^j(d\theta, d\eta).$$
From \eqref{eq2018-03-01-3},
\begin{align}\label{eq2018-09-10-11}
|u_i(t, x) - u_i(s, y)| &\leq \sum_{j = 1}^d|I_{ij}(t, x) - I_{ij}(s, x)| + \sum_{j = 1}^d|I_{ij}(s, x) - I_{ij}(s, y)| \nonumber \\
& \quad + \int_0^{t - s}\int_{\mathbb{R}^k}S(t - \theta, x - \eta)|b_i(u(\theta, \eta))|d\eta d\theta \nonumber \\
& \quad + \int_0^{s}\int_{\mathbb{R}^k}|S(t - \theta, x - \eta) - S(t - \theta, y - \eta)|\, |b_i(u(\theta, \eta))|d\eta d\theta \nonumber \\
& \quad + \int_0^{s}\int_{\mathbb{R}^k}|S(t - \theta, y - \eta) - S(s - \theta, y - \eta)|\, |b_i(u(\theta, \eta))|d\eta d\theta \nonumber \\
& := I_1 + I_2 + I_3 + I_4+ I_5.
\end{align}
By Burkholder's inequality, for any $p \geq 2$,
 \begin{align*}
\mbox{E}[|I_1|^p] &\leq c \sum_{j = 1}^d \mbox{E}\bigg[\Big|\int_s^td\theta\int_{\mathbb{R}^k}dz\int_{\mathbb{R}^k}dv \, \frac{1}{\|z - v\|^{\beta}} \nonumber \\
&\qquad \qquad \qquad\times S(t - \theta, x - z)\, |\sigma_{ij}(u(\theta, z))|\, S(t - \theta, x - v)\, |\sigma_{ij}(u(\theta, v))|\Big|^{\frac{p}{2}}\bigg]\nonumber \\
& \quad + c \sum_{j = 1}^d \mbox{E}\bigg[\Big|\int_0^sd\theta\int_{\mathbb{R}^k}dz\int_{\mathbb{R}^k}dv \, \frac{1}{\|z - v\|^{\beta}}|S(t - \theta, x - z) - S(s - \theta, x - z)| \nonumber \\
&\qquad \qquad \qquad\times |\sigma_{ij}(u(\theta, z))|\, |S(t - \theta, x - v) - S(s - \theta, x - v)|\, |\sigma_{ij}(u(\theta, v))|\Big|^{\frac{p}{2}}\bigg].
\end{align*}
Using Minkowski inequality and the Cauchy-Schwartz inequality, \eqref{eq2018-09-10-1} and the linear growth property of the functions $\sigma_{ij}$, this is bounded above by
\begin{align*}
 & c \bigg[\int_s^td\theta\int_{\mathbb{R}^k}dz\int_{\mathbb{R}^k}dv \, \frac{1}{\|z - v\|^{\beta}} S(t - \theta, x - z)S(t - \theta, x - v)\bigg]^{\frac{p}{2}}\nonumber \\
& \quad + c \bigg[\int_0^sd\theta\int_{\mathbb{R}^k}dz\int_{\mathbb{R}^k}dv \, \frac{1}{\|z - v\|^{\beta}}|S(t - \theta, x - z) - S(s - \theta, x - z)| \nonumber \\
&\qquad \qquad \qquad\qquad \qquad\times  |S(t - \theta, x - v) - S(s - \theta, x - v)|\bigg]^{\frac{p}{2}}.
\end{align*}
The first term above is equal to $c(t - s)^{(2 - \beta)p/4}$ by \cite[(6.3)]{DKN13} and the second term above is bounded above by $c(t - s)^{(2 - \beta)p/4}$ by Lemma \ref{lemma2018-09-10-5}. Hence for any $p \geq 2$,
\begin{align}\label{eq2018-09-10-12}
 \mbox{E}[|I_1|^p] &\leq c \,(t - s)^{\frac{(2 - \beta)p}{4}}.
\end{align}

Similarly, applying Burkholder's inequality and taking the absolute value inside,
\begin{align*}
\mbox{E}[|I_2|^p] &\leq c \sum_{j = 1}^d \mbox{E}\bigg[\Big|\int_0^sdr\int_{\mathbb{R}^k}dz\int_{\mathbb{R}^k}dv \, \|z - v\|^{-\beta}|S(r, x - z) - S(r, y - z)|\, |\sigma_{ij}(u(s-r, z))| \nonumber \\
& \qquad\qquad\qquad \times  |S(r, x - v) - S(r, y - v)|\, |\sigma_{ij}(u(s-r, v))| \Big|^{p/2}\bigg].
\end{align*}
By the Minkowski inequality with respect to the measure $\|z - v\|^{-\beta}|S(r, x - z) - S(r, y - z)|\,|S(r, x - v) - S(r, y - v)|drdvdz$, the Cauchy-Schwartz inequality, \eqref{eq2018-09-10-1} and the linear growth property of the functions $\sigma_{ij}$, this is bounded above by
\begin{align}\label{eq2018-09-10-13}
 &c\sup_{(t, x) \in [0, T] \times \mathbb{R}^k}\mbox{E}\left[1 + \|u(t, x)\|^p\right]\Big|\int_0^sdr\int_{\mathbb{R}^k}dz\int_{\mathbb{R}^k}dv \, \|z - v\|^{-\beta} \nonumber \\
 &\qquad \qquad\qquad\times   |S(r, x - z) - S(r, y - z)| \, |S(r, x - v) - S(r, y - v)| \Big|^{p/2} \nonumber \\
 & \quad \leq c \|x - y\|^{\frac{(2 - \beta)p}{2}},
\end{align}
where the inequality follows from Lemma \ref{lemma2018-09-10-4}.

For the estimate of $I_3$, using the Minkowski inequality with respect to the measure $S(t - \theta, x - \eta)d\eta d\theta$, \eqref{eq2018-09-10-1} and the linear growth property of the functions $b_{i}$, we have
\begin{align}\label{eq2018-09-10-14}
 \mbox{E}[|I_3|^p] &\leq c\sup_{(t, x) \in [0, T] \times \mathbb{R}^k}\mbox{E}\left[1 + \|u(t, x)\|^p\right]\Big(\int_0^{t - s}\int_{\mathbb{R}^k}S(t - \theta, x - \eta)d\eta d\theta\Big)^p \nonumber \\
 & = c\, |t - s|^p.
\end{align}
Moreover, using the Minkowski inequality with respect to the measure $|S(t - \theta, x - \eta) - S(t - \theta, y - \eta)|d\eta d\theta$, \eqref{eq2018-09-10-1} and the linear growth property of the functions $b_{i}$,
\begin{align}\label{eq2018-09-10-15}
 \mbox{E}[|I_4|^p] &\leq c\, \Big(\int_0^{s}\int_{\mathbb{R}^k}|S(t - \theta, x - \eta) - S(t - \theta, y - \eta)|d\eta d\theta\Big)^p \nonumber \\
 & \leq c\, \|x - y\|^p,
\end{align}
where the second inequality follows from \cite[Lemme A2]{Sai98}.

Similarly, by the Minkowski inequality with respect to the measure $|S(t - \theta, y - \eta) - S(s - \theta, y - \eta)|d\eta d\theta$, \eqref{eq2018-09-10-1} and the linear growth property of the functions $b_{i}$,
\begin{align}\label{eq2018-09-10-1500}
 \mbox{E}[|I_5|^p] &\leq c\, \Big(\int_0^{s}\int_{\mathbb{R}^k}|S(t - \theta, y - \eta) - S(s - \theta, y - \eta)|d\eta d\theta\Big)^p \nonumber \\
 & \leq c\, |(t - s)\log(t - s)|^p,
\end{align}
where the second inequality follows from \cite[Lemme A3]{Sai98}.

Hence, \eqref{eq2018-09-10-11}--\eqref{eq2018-09-10-1500} imply \eqref{eq2018-09-10-10}.
\end{proof}

\begin{proof}[Proof of Theorem \ref{prop2018-09-11-1}]
The proof is similar to that of Theorem \ref{prop2018-09-10-1} by using Lemmas \ref{lemma2018-09-10-4} and \ref{lemma2018-09-10-5}. From \eqref{eq2018-09-10-2}, we assume that $t - s \leq 1/2$ and $\|x - y\| \leq 1/2$.

We assume $m = 1$ and fix $p \geq 2$. Let
\begin{align*}
g_{t, x; s, y}(r, *) := S(t - r, x - *)1_{\{r < t\}} - S(s - r, y - *)1_{\{r < s\}}.
\end{align*}
Using \eqref{eq2018-03-01-4}, we see that
\begin{align*}
\|D(u_i(t, x) - u_i(s, y))\|^p_{L^p(\Omega; \mathscr{H}_T^d)} \leq c_p(A_{1,1} + A_{1,2} + A_{1,3} + A_{2,1} + A_{2,2} + A_{2,3} + A_{3,1} + A_{3,2}),
\end{align*}
where
\begin{align*}
A_{1,1} &= \mbox{E}\bigg[\Big(\int_s^tdr \sum_{j = 1}^d\|S(t - r, x - *)\sigma_{ij}(u(r, *))\|_{\mathscr{H}}^2\Big)^{p/2}\bigg], \\
A_{1,2} &= \mbox{E}\bigg[\Big(\int_0^sdr \sum_{j = 1}^d\|g_{t, x; s, x}(r, *)\sigma_{ij}(u(r, *))\|_{\mathscr{H}}^2\Big)^{p/2}\bigg], \\
A_{1,3} &= \mbox{E}\bigg[\Big(\int_0^sdr \sum_{j = 1}^d\|g_{s, x; s, y}(r, *)\sigma_{ij}(u(r, *))\|_{\mathscr{H}}^2\Big)^{p/2}\bigg], \\
A_{2,1} &= \mbox{E}\bigg[\Big\|\int_s^t\int_{\mathbb{R}^k}S(t - \theta, x - \eta)\sum_{j = 1}^dD(\sigma_{ij}(u(\theta, \eta)))M^j(d\theta, d\eta)\Big\|^p_{\mathscr{H}_T^d}\bigg], \\
A_{2,2} &= \mbox{E}\bigg[\Big\|\int_0^s\int_{\mathbb{R}^k}g_{t, x; s, x}(\theta, \eta)\sum_{j = 1}^dD(\sigma_{ij}(u(\theta, \eta)))M^j(d\theta, d\eta)\Big\|^p_{\mathscr{H}_T^d}\bigg], \\
A_{2,3} &= \mbox{E}\bigg[\Big\|\int_0^s\int_{\mathbb{R}^k}g_{s, x; s, y}(\theta, \eta)\sum_{j = 1}^dD(\sigma_{ij}(u(\theta, \eta)))M^j(d\theta, d\eta)\Big\|^p_{\mathscr{H}_T^d}\bigg], \\
A_{3,1} &= \mbox{E}\bigg[\Big\|\int_0^{t - s}\int_{\mathbb{R}^k}S(t - \theta, x - \eta)D(b_{i}(u(\theta, \eta)))d\theta d\eta\Big\|^p_{\mathscr{H}_T^d}\bigg], \\
A_{3,2} &= \mbox{E}\bigg[\Big\|\int_0^s\int_{\mathbb{R}^k}(S(t - \theta, x - \eta) - S(t - \theta, y - \eta))\,D(b_{i}(u(\theta, \eta)))d\theta d\eta\Big\|^p_{\mathscr{H}_T^d}\bigg], \\
A_{3,3} &= \mbox{E}\bigg[\Big\|\int_0^s\int_{\mathbb{R}^k}(S(t - \theta, y - \eta) - S(s - \theta, y - \eta))\,D(b_{i}(u(\theta, \eta)))d\theta d\eta\Big\|^p_{\mathscr{H}_T^d}\bigg].
\end{align*}

Using the Minkowski inequality, the Cauchy-Schwartz inequality, \eqref{eq2018-09-10-1} and the linear growth property of the functions $\sigma_{ij}$, we have
\begin{align}\label{eq2018-09-11-4}
A_{1, 1} &\leq c \sup_{(t, x) \in [0, T] \times \mathbb{R}^k}\mbox{E}\left[1 + \|u(t, x)\|^p\right]\Big(\int_s^tdr \int_{\mathbb{R}^k}\mu(d\xi)|\mathscr{F}S(t - r, x - *)(\xi)|^2\Big)^{p/2} \nonumber \\
& = c \Big(\int_s^t (t - r)^{-\beta/2}dr\Big)^{p/2} = c'(t - s)^{\frac{(2 - \beta)p}{4}},
\end{align}
where the first equality is due to \cite[(6.3)]{DKN13}.

Similarly,
\begin{align}\label{eq2018-09-11-5}
A_{1, 2} &\leq c \sup_{(t, x) \in [0, T] \times \mathbb{R}^k}\mbox{E}\left[1 + \|u(t, x)\|^p\right]\Big(\int_0^sdr \int_{\mathbb{R}^k}dz\int_{\mathbb{R}^k}dv \, \|z - v\|^{-\beta} \nonumber \\
& \qquad \times |S(t - r, x - z) - S(s - r, x - z)| \, |S(t - r, x - v) - S(s - r, x - v)| \Big)^{p/2} \nonumber \\
& \leq c (t - s)^{\frac{(2 - \beta)p}{4}},
\end{align}
where the last inequality follows from \eqref{eq2018-09-10-1} and Lemma \ref{lemma2018-09-10-5}.

Moreover, by the Minkowski inequality,  the Cauchy-Schwartz inequality, \eqref{eq2018-09-10-1} and the linear growth property of $\sigma_{ij}$, we have
\begin{align}\label{eq2018-09-11-5}
A_{1, 3} &\leq c \sup_{(t, x) \in [0, T] \times \mathbb{R}^k}\mbox{E}\left[1 + \|u(t, x)\|^p\right]\Big(\int_0^sdr \int_{\mathbb{R}^k}dz\int_{\mathbb{R}^k}dv \, \|z - v\|^{-\beta} \nonumber \\
& \qquad \times |S(s - r, x - z) - S(s - r, y - z)| \, |S(s - r, x - v) - S(s - r, y - v)| \Big)^{p/2} \nonumber \\
& \leq c \|x - y\|^{\frac{(2 - \beta)p}{2}},
\end{align}
where the last inequality follows from \eqref{eq2018-09-10-1} and Lemma \ref{lemma2018-09-10-4}.

The estimate of $A_{2, 1}$ is similar to that of $A_{1, 1}$. Indeed, by Burkholder's inequality for Hilbert-space-valued martingales (\cite[E.2. p.212]{MeM82}),
\begin{align*}
A_{2, 1} &\leq c \sum_{j = 1}^d \mbox{E}\bigg[\Big(\int_s^tdr\int_{\mathbb{R}^k}dz\int_{\mathbb{R}^k}dv  \, \|z - v\|^{-\beta}S(t - r, x - z)\|D(\sigma_{ij}(u(\theta, z)))\|_{\mathscr{H}_T^d}\nonumber \\
& \qquad \qquad \qquad \qquad \times S(t - r, x - v)\|D(\sigma_{ij}(u(\theta, v)))\|_{\mathscr{H}_T^d}\Big)^{p/2}\bigg]
\end{align*}
By hypothesis \textbf{P1}, the Minkowski inequality, the Cauchy-Schwarz inequality and \eqref{eq2018-09-10-2}, this is bounded above by
\begin{align}\label{eq2018-09-11-6}
& c \sum_{l = 1}^d\sup_{(t, x) \in [0, T] \times \mathbb{R}^k}\mbox{E}\left[\|D(u_l(t, x))\|^p_{\mathscr{H}_T^d}\right]\nonumber \\
& \quad \quad \times \Big(\int_s^tdr\int_{\mathbb{R}^k}dz\int_{\mathbb{R}^k}dv  \, \|z - v\|^{-\beta}S(t - r, x - z)S(t - r, x - v)\Big)^{p/2} \nonumber \\
& \quad = c \Big(\int_s^tdr \int_{\mathbb{R}^k}\mu(d\xi)|\mathscr{F}S(t - r, x - *)(\xi)|^2\Big)^{p/2} \nonumber \\
&\quad  = c \Big(\int_s^t (t - r)^{-\beta/2}dr\Big)^{p/2} = c'(t - s)^{\frac{(2 - \beta)p}{4}},
\end{align}
where the second equality is due to \cite[(6.3)]{DKN13}.

Furthermore,  by Burkholder's inequality for Hilbert-space-valued martingales (\cite[E.2. p.212]{MeM82}),
\begin{align*}
A_{2, 2} &\leq c \sum_{j = 1}^d \mbox{E}\bigg[\Big(\int_0^sdr\int_{\mathbb{R}^k}dz\int_{\mathbb{R}^k}dv  \, \|z - v\|^{-\beta}|S(t - r, x - z) - S(s - r, x - z)|\nonumber \\
& \quad   \times \|D(\sigma_{ij}(u(\theta, z)))\|_{\mathscr{H}_T^d} |S(t - r, x - v) - S(s - r, x - v)| \, \|D(\sigma_{ij}(u(\theta, v)))\|_{\mathscr{H}_T^d}\Big)^{p/2}\bigg].
\end{align*}
Similar to the estimate of $A_{1, 2}$, by hypothesis \textbf{P1}, the Minkowski inequality, the Cauchy-Schwarz inequality and \eqref{eq2018-09-10-2}, this is bounded above by
\begin{align}\label{eq2018-09-11-7}
& c \sum_{l = 1}^d\sup_{(t, x) \in [0, T] \times \mathbb{R}^k}\mbox{E}\left[\|D(u_l(t, x))\|^p_{\mathscr{H}_T^d}\right]\Big(\int_0^sdr\int_{\mathbb{R}^k}dz\int_{\mathbb{R}^k}dv  \, \|z - v\|^{-\beta}\nonumber \\
& \quad \quad \qquad\qquad \times |S(t - r, x - z) - S(s - r, x - z)| \, |S(t - r, x - v) - S(s - r, x - v)|\Big)^{p/2} \nonumber \\
&\quad = c \Big(\int_0^sdr \int_{\mathbb{R}^k}dz\int_{\mathbb{R}^k}dv \, \|z - v\|^{-\beta}|S(t - s + r, x - z) - S(r, x - z)| \nonumber \\
& \qquad \qquad\qquad \qquad\qquad \qquad \times  \, |S(t -s + r, x - v) - S(r, x - v)| \Big)^{p/2} \nonumber \\
&\quad \leq c (t - s)^{\frac{(2 - \beta)p}{4}},
\end{align}
where the last inequality follows from Lemma \ref{lemma2018-09-10-5}.

We move on to estimate $A_{2, 3}$. By Burkholder's inequality for Hilbert-space-valued martingales (\cite[E.2. p.212]{MeM82}),
\begin{align*}
A_{2, 3} &\leq c \sum_{j = 1}^d \mbox{E}\bigg[\Big(\int_0^sdr\int_{\mathbb{R}^k}dz\int_{\mathbb{R}^k}dv  \, \|z - v\|^{-\beta}|S(s - r, x - z) - S(s - r, y - z)|\nonumber \\
& \quad   \times \|D(\sigma_{ij}(u(\theta, z)))\|_{\mathscr{H}_T^d} |S(s - r, x - v) - S(s - r, y - v)| \, \|D(\sigma_{ij}(u(\theta, v)))\|_{\mathscr{H}_T^d}\Big)^{p/2}\bigg].
\end{align*}
Again, using hypothesis \textbf{P1}, the Minkowski inequality, the Cauchy-Schwarz inequality and \eqref{eq2018-09-10-2}, this is bounded above by
\begin{align}\label{eq2018-09-11-7}
& c \sum_{l = 1}^d\sup_{(t, x) \in [0, T] \times \mathbb{R}^k}\mbox{E}\left[\|D(u_l(t, x))\|^p_{\mathscr{H}_T^d}\right]\Big(\int_0^sdr\int_{\mathbb{R}^k}dz\int_{\mathbb{R}^k}dv  \, \|z - v\|^{-\beta}\nonumber \\
& \quad \quad \qquad\qquad \times |S(s - r, x - z) - S(s - r, y - z)| \, |S(s - r, x - v) - S(s - r, y - v)|\Big)^{p/2} \nonumber \\
&\quad = c \Big(\int_0^sdr \int_{\mathbb{R}^k}dz\int_{\mathbb{R}^k}dv \, \|z - v\|^{-\beta}|S(r, x - z) - S(r, y - z)| \nonumber \\
& \qquad \qquad\qquad \qquad\qquad \qquad \times  \, |S(r, x - v) - S(r, y - v)| \Big)^{p/2} \nonumber \\
&\quad \leq c \|x - y\|^{\frac{(2 - \beta)p}{2}},
\end{align}
where the last inequality follows from Lemma \ref{lemma2018-09-10-4}.

We proceed to estimate $A_{3, 1}$, $A_{3, 2}$ and $A_{3, 3}$. For $A_{3, 1}$, by hypothesis \textbf{P1}, the Minkowski inequality and \eqref{eq2018-09-10-2},
\begin{align}\label{eq2018-09-11-8}
A_{3, 1} & \leq c \sum_{l = 1}^d\sup_{(t, x) \in [0, T] \times \mathbb{R}^k}\mbox{E}\left[\|D(u_l(t, x))\|^p_{\mathscr{H}_T^d}\right]\Big(\int_0^{t - s}\int_{\mathbb{R}^k}S(t - \theta, x - \eta)d\theta d\eta\Big)^p \nonumber \\
& = c (t - s)^p.
\end{align}

Similar to the estimate of $I_4$ in the proof of Theorem \ref{prop2018-09-10-1}, by hypothesis \textbf{P1} and the Minkowski inequality,
\begin{align}\label{eq2018-09-11-800}
A_{3, 2} & \leq c \sum_{l = 1}^d\sup_{(t, x) \in [0, T] \times \mathbb{R}^k}\mbox{E}\left[\|D(u_l(t, x))\|^p_{\mathscr{H}_T^d}\right]\nonumber \\
& \qquad \qquad \times \Big(\int_0^{s}\int_{\mathbb{R}^k}|S(t - \theta, x - \eta) - S(t - \theta, y - \eta)|d\eta d\theta\Big)^p \nonumber \\
 & \leq c\, \|x - y\|^p,
\end{align}
where the second inequality follows from \eqref{eq2018-09-10-2} and  \cite[Lemme A2]{Sai98}.

Moreover, by hypothesis \textbf{P1} and the Minkowski inequality,
\begin{align}\label{eq2018-09-11-80011}
A_{3, 3} & \leq c \sum_{l = 1}^d\sup_{(t, x) \in [0, T] \times \mathbb{R}^k}\mbox{E}\left[\|D(u_l(t, x))\|^p_{\mathscr{H}_T^d}\right]\nonumber \\
& \qquad \qquad \times \Big(\int_0^{s}\int_{\mathbb{R}^k}|S(t - \theta, y - \eta) - S(s - \theta, y - \eta)|d\eta d\theta\Big)^p \nonumber \\
 & \leq c\, |(t - s)\log(t - s)|^p,
\end{align}
where the second inequality follows from \eqref{eq2018-09-10-2} and \cite[Lemme A3]{Sai98}.

Therefore, \eqref{eq2018-09-11-4}--\eqref{eq2018-09-11-80011} together prove \eqref{eq2018-09-11-1} for $m = 1$.
The case $m > 1$ follows along the same lines by induction using \eqref{eq2018-09-10-2} and the stochastic partial differential equations satisfied by the iterated derivatives (cf. \cite[Proposition 6.1]{NuQ07}).
\end{proof}

\section{Proof of Theorem \ref{prop2017-09-19-4000}} \label{section2018-09-21-1}

We first state an elementary fact that will be used several times later on.
\begin{lemma}\label{lemma2018-11-22-2}
Fix $\gamma \in \, ]0, 1[$ and $\mu > 0$.
\begin{itemize}
  \item [(a)] The function $x \mapsto (x + \mu)^{\gamma} - x^{\gamma}$ is nonincreasing on $[0, \infty[$ and the function $x \mapsto x^{\gamma} - (x - \mu)^{\gamma}$ is nonincreasing on $[\mu, \infty[$.
  \item [(b)] $(1 + x)^{\gamma} - 1 \leq \gamma x$ for all $x \geq 0$.
\end{itemize}
\end{lemma}

We recall from \cite[p.148]{DKN13} an estimate on the Malliavin derivative of the solution.
\begin{lemma}\label{lemma2018-11-22-1}
For all $q \geq 1$, $0 < \epsilon \leq s \leq T$ and $s - \epsilon \leq \rho \leq T$, there exists $C > 0$ such that for all $i \in \{1, \ldots, d\}$,
\begin{align}
\sup_{x \in \mathbb{R}^k}\mathrm{E}\left[\|D_{\cdot, *}(u_i(\rho, x))\|_{\mathscr{H}_{s - \epsilon, s}^d}^{2q}\right] & \leq C\Big(\int_{s - \epsilon}^{s \wedge \rho}d\theta \int_{\mathbb{R}^k}\mu(d\xi)|\mathscr{F}S(\rho - \theta, *)(\xi)|^2\Big)^q \label{eq2018-09-12-5}\\
& \leq C \epsilon^{\frac{(2 - \beta)q}{2}}.\label{eq2018-09-12-6}
\end{align}
\end{lemma}
\noindent{}Note that \eqref{eq2018-09-12-5} is exactly the estimate between $(6.2)$ and $(6.3)$ in \cite[p.148]{DKN13} and \eqref{eq2018-09-12-6} follows from the calculation below \cite[(6.3)]{DKN13}. We give the proof of \eqref{eq2018-09-12-5} in the appendix for reader's convenience.

We next give an estimate on $a_i(l, r, t, x)$, which is a refinement of \cite[Lemma 6.2]{DKN13}.

\begin{lemma} \label{lemma2018-09-12-1}
Assume \textbf{P1}. Fix $T > 0$, $c_0 > 1$ and $0 < \gamma_0 < 1$. For any $q \geq 1$, there exists a constant $c = c(c_0, \gamma_0, q, T) >0$ such that for every $0 <\epsilon \leq s \leq t \leq T$ with $t - s > c_0 \epsilon^{\gamma_0}$ and $x \in \mathbb{R}^k$,
\begin{align}\label{eq2018-09-12-2}
&W:= \mathrm{E}\bigg[\sup_{\xi \in \mathbb{R}^d: \|\xi\| \leq 1} \Big(\int_{s - \epsilon}^s dr \sum_{l = 1}^d\Big\|\sum_{i = 1}^da_i(l, r, t, x)\xi_i\Big\|_{\mathscr{H}}^2\Big)^q\bigg] \nonumber \\
&\qquad \leq c\, \epsilon^{\min\left\{\frac{2 - \beta}{2}(1 + \gamma_0), \, 1 - \frac{\beta\gamma_0}{2}\right\}q},
\end{align}
where $a_i(l, r, t, x)$ is defined in \eqref{eq2018-09-12-1}.
\end{lemma}

\begin{proof}
We adopt the same notation as in the proof of \cite[Lemma 6.2]{DKN13}. Use \eqref{eq2018-09-12-1} and the Cauchy-Schwartz inequality to get
\begin{align}\label{eq2018-09-12-3}
W \leq c\, \bigg(\Big(\sum_{i, j = 1}^d \mbox{E}\Big[\Big(\int_{s - \epsilon}^sdr \, \|W_1\|_{\mathscr{H}^d}^2\Big)^q\Big]\Big) + \Big(\sum_{i = 1}^d\mbox{E}\Big[\Big(\int_{s - \epsilon}^sdr \, \|W_2\|_{\mathscr{H}^d}^2\Big)^q\Big] \Big)\bigg),
\end{align}
where
\begin{align*}
W_1 &= \int_r^t\int_{\mathbb{R}^k}S(t - \theta, x - \eta)D_r(\sigma_{ij}(u(\theta, \eta)))M^j(d\theta, d\eta), \\
W_2 &=  \int_r^td\theta\int_{\mathbb{R}^k}d\eta \, S(t - \theta, x - \eta)D_r(b_{i}(u(\theta, \eta))).
\end{align*}
Then
\begin{align*}
\mbox{E}\Big[\Big(\int_{s - \epsilon}^sdr\, \|W_1\|_{\mathscr{H}^d}^2\Big)^q\Big] = \mbox{E}\left[\|W_1\|^{2q}_{\mathscr{H}_{s-\epsilon, s}^d}\right].
\end{align*}
Using hypothesis \textbf{P1}, Burkholder's inequality for Hilbert-space-valued martingales (\cite[E.2. p.212]{MeM82}) ensures that this is
\begin{align}\label{eq2018-09-12-4}
&\leq c\sum_{l = 1}^d \mbox{E}\Big[\Big(\int_{s - \epsilon}^tdr \int_{\mathbb{R}^k}dz\int_{\mathbb{R}^k}dv \, \|z - v\|^{-\beta} S(t - r, x - z)\|D_{\cdot, *}(u_l(r, z))\|_{\mathscr{H}_{s-\epsilon, s}^d}\nonumber \\
&\quad \quad\qquad\qquad\qquad \times  S(t - r, x - v)\|D_{\cdot, *}(u_l(r, v))\|_{\mathscr{H}_{s-\epsilon, s}^d}\Big)^q\Big]\nonumber \\
& \quad \leq I_{1,1} + I_{1,2} + I_{1,3},
\end{align}
where for $i = 1, 2, 3$,
\begin{align*}
I_{1,i} &:= c\sum_{l = 1}^d \mbox{E}\Big[\Big(\int_{a_i}^{b_i}dr \int_{\mathbb{R}^k}dz\int_{\mathbb{R}^k}dv \, \|z - v\|^{-\beta} S(t - r, x - z)\|D_{\cdot, *}(u_l(r, z))\|_{\mathscr{H}_{s-\epsilon, s}^d}\nonumber \\
&\quad \quad\qquad\qquad\qquad \times  S(t - r, x - v)\|D_{\cdot, *}(u_l(r, v))\|_{\mathscr{H}_{s-\epsilon, s}^d}\Big)^q\Big],
\end{align*}
and
\begin{align}\label{eq2018-11-22-20}
a_1 = s - \epsilon, \,\,b_1 = s,\,\, a_2 = s, \,\,b_2 = s + c_0\epsilon^{\gamma_0},\,\, a_3 = s + c_0\epsilon^{\gamma_0},\,\, b_3 = t.
\end{align}

We now estimate $I_{1, i}$. Applying H\"{o}lder's inequality and the Cauchy-Schwartz inequality,
\begin{align}\label{eq2018-11-22-19}
I_{1,i} &\leq c\sum_{l = 1}^d\Big(\int_{a_i}^{b_i}dr\int_{\mathbb{R}^k}\mu(d\xi)|\mathscr{F}S(t - r, x - *)(\xi)|^2\Big)^{q - 1} \nonumber \\
& \quad \times \int_{a_i}^{b_i}dr\int_{\mathbb{R}^k}\mu(d\xi)|\mathscr{F}S(t - r, x - *)(\xi)|^2\sup_{\eta \in \mathbb{R}^k}\mbox{E}\left[\|D_{\cdot, *}(u_l(r, \eta))\|_{\mathscr{H}_{s - \epsilon, s}^d}^{2q}\right].
\end{align}
In the case $i = 1$, we find that, by \eqref{eq2018-09-12-6} and \cite[(6.3)]{DKN13},
\begin{align}\label{eq2018-09-12-7}
I_{1, 1}& \leq c\, \epsilon^{\frac{(2 - \beta)q}{2}} \Big(\int_{s - \epsilon}^sdr\int_{\mathbb{R}^k}\mu(d\xi)|\mathscr{F}S(t - r, x - *)(\xi)|^2\Big)^{q} \nonumber \\
& = c\, \epsilon^{\frac{(2 - \beta)q}{2}}\big((t - s + \epsilon)^{(2 - \beta)/2} -  (t - s)^{(2 - \beta)/2}\big)^q \nonumber \\
& \leq c\, \epsilon^{\frac{(2 - \beta)q}{2}}\big((c_0\epsilon^{\gamma_0} + \epsilon)^{(2 - \beta)/2} -  (c_0\epsilon^{\gamma_0})^{(2 - \beta)/2}\big)^q \nonumber \\
&  = c\, \epsilon^{\frac{(2 - \beta)(1 + \gamma_0)q}{2}}\big((1 + c_0^{-1}\epsilon^{1 - \gamma_0})^{(2 - \beta)/2} -  1\big)^q \nonumber \\
& \leq  c\, \epsilon^{\frac{(2 - \beta)(1 + \gamma_0)q}{2} + (1 - \gamma_0)q},
\end{align}
where the second inequality follows from Lemma \ref{lemma2018-11-22-2}(a) because $t - s > c_0\epsilon^{\gamma_0}$, and the third inequality is due to Lemma \ref{lemma2018-11-22-2}(b).

Similarly, by \eqref{eq2018-11-22-19},
\begin{align}\label{eq2018-09-12-8}
I_{1,2} &\leq c\, \epsilon^{\frac{(2 - \beta)q}{2}}\big((t - s)^{(2 - \beta)/2} -  (t - s - c_0\epsilon^{\gamma_0})^{(2 - \beta)/2}\big)^q \nonumber \\
& \leq c\, \epsilon^{\frac{(2 - \beta)q}{2}}(c_0\epsilon^{\gamma_0})^{(2 - \beta)q/2} = c'\epsilon^{\frac{2 - \beta}{2}(1 + \gamma_0)q},
\end{align}
where the second inequality holds by Lemma \ref{lemma2018-11-22-2}(a) since $t - s > c_0\epsilon^{\gamma_0}$.

Moreover, by \eqref{eq2018-11-22-19}, using \cite[(6.3)]{DKN13} and \eqref{eq2018-09-12-5},
\begin{align} \label{eq2018-09-13-1}
I_{1,3} &\leq c\, (t - s - c_0\epsilon^{\gamma_0})^{\frac{(2 - \beta)(q - 1)}{2}} \int_{s + c_0\epsilon^{\gamma_0}}^tdr\int_{\mathbb{R}^k}\mu(d\xi)|\mathscr{F}S(t - r, x - *)(\xi)|^2 \nonumber \\
 & \qquad\qquad\qquad\qquad\qquad\qquad\qquad\qquad \times \Big(\int_{s - \epsilon}^{s}d\theta \int_{\mathbb{R}^k}\mu(d\xi)|\mathscr{F}S(r - \theta, *)(\xi)|^2\Big)^q \nonumber \\
  &= c\, (t - s - c_0\epsilon^{\gamma_0})^{\frac{(2 - \beta)(q - 1)}{2}} \int_{s + c_0\epsilon^{\gamma_0}}^tdr\int_{\mathbb{R}^k}\mu(d\xi)|\mathscr{F}S(t - r, x - *)(\xi)|^2 \nonumber \\
 & \qquad\qquad\qquad\qquad\qquad\qquad\qquad\qquad \times \big((r - s + \epsilon)^{\frac{2 - \beta}{2}} -  (r - s)^{\frac{2 - \beta}{2}}\big)^q \nonumber \\
 &  \leq c\, (t - s - c_0\epsilon^{\gamma_0})^{\frac{(2 - \beta)q}{2}} \big((c_0\epsilon^{\gamma_0} + \epsilon)^{\frac{2 - \beta}{2}} -  (c_0\epsilon^{\gamma_0})^{\frac{2 - \beta}{2}}\big)^q \nonumber \\
 &   \leq c\, \epsilon^{\frac{2 - \beta}{2}\gamma_0q}\big((1 + c_0^{-1}\epsilon^{1 - \gamma_0})^{\frac{2 - \beta}{2}} -  1\big)^q \nonumber \\
 &   \leq c\,\epsilon^{(\frac{2 - \beta}{2}\gamma_0 + 1 - \gamma_0)q} =  c\,\epsilon^{(1 - \beta\gamma_0/2)q},
\end{align}
where in the equality we use \cite[Lemma 6.1]{DKN13}, the second inequality follows from Lemma \ref{lemma2018-11-22-2}(a) since $r - s \geq c_0\epsilon^{\gamma_0}$ for all $r \in [s + c_0\epsilon^{\gamma_0}, t]$, in the third inequality we bound $t - s - c_0\epsilon^{\gamma_0}$ by $T$, and in the fourth inequality we use Lemma \ref{lemma2018-11-22-2}(b).

We proceed to estimate the second term in \eqref{eq2018-09-12-3}. First, by hypothesis \textbf{P1} and the Cauchy-Schwartz inequality,
\begin{align*}
\int_{s - \epsilon}^sdr \, \|W_2\|_{\mathscr{H}^d}^2 &\leq c \sum_{l = 1}^d\int_{s - \epsilon}^s dr\Big(\int_r^td\rho\int_{\mathbb{R}^k}d\xi \, S(t - \rho, x - \xi)\|D_r(u_l(\rho, \xi))\|_{\mathscr{H}^d}\Big)^2 \\
&\leq c \sum_{l = 1}^d(t - s + \epsilon)\int_{s - \epsilon}^sdr\int_r^td\rho\int_{\mathbb{R}^k}d\xi \, S(t - \rho, x - \xi)\|D_r(u_l(\rho, \xi))\|^2_{\mathscr{H}^d} \\
& \leq c \sum_{l = 1}^d\int_{s - \epsilon}^td\rho\int_{\mathbb{R}^k}d\xi \, S(t - \rho, x - \xi)\|D_{\cdot, *}(u_l(\rho, \xi))\|^2_{\mathscr{H}_{s - \epsilon, s}^d}.
\end{align*}
Therefore,
\begin{align*}
\mbox{E}\Big[\Big(\int_{s - \epsilon}^sdr \, \|W_2\|_{\mathscr{H}^d}^2\Big)^q\Big] \leq c\,(I_{2,1} + I_{2,2} + I_{2,3}),
\end{align*}
where for $i = 1, 2, 3$,
\begin{align*}
I_{2,i} &= \sum_{l = 1}^d\mbox{E}\Big[\Big(\int_{a_i}^{b_i}d\rho\int_{\mathbb{R}^k}d\xi \, S(t - \rho, x - \xi)\|D_{\cdot, *}(u_l(\rho, \xi))\|^2_{\mathscr{H}_{s - \epsilon, s}^d}\Big)^q\Big],
\end{align*}
and $a_i$, $b_i$ are defined in \eqref{eq2018-11-22-20}.
By H\"{o}lder's inequality,
\begin{align} \label{eq2018-11-23-1}
I_{2,i} &\leq c  \sum_{l = 1}^d\Big(\int_{a_i}^{b_i}d\rho\int_{\mathbb{R}^k}d\xi \, S(t - \rho, x - \xi)\Big)^{q - 1} \nonumber \\
& \qquad \times \int_{a_i}^{b_i}d\rho\int_{\mathbb{R}^k}d\xi \, S(t - \rho, x - \xi)\mbox{E}\left[\|D_{\cdot, *}(u_l(\rho, \xi))\|_{\mathscr{H}_{s - \epsilon, s}^d}^{2q}\right].
\end{align}
In the case $i = 1$, by  \eqref{eq2018-09-12-6},
\begin{align}\label{eq2018-09-13-2}
I_{2,1} &\leq c \, \epsilon^{\frac{(2 - \beta)q}{2}}\Big(\int_{s - \epsilon}^{s}d\rho\int_{\mathbb{R}^k}d\xi \, S(t - \rho, x - \xi)\Big)^{q} = c \, \epsilon^{(2 - \beta/2)q}.
\end{align}
Similarly, by \eqref{eq2018-11-23-1},
\begin{align}\label{eq2018-09-13-3}
I_{2,2} &\leq c \, \epsilon^{\frac{(2 - \beta)q}{2}}\Big(\int_{s}^{s + c_0\epsilon^{\gamma_0}}d\rho\int_{\mathbb{R}^k}d\xi \, S(t - \rho, x - \xi)\Big)^{q} \nonumber \\
& = c\, \epsilon^{\gamma_0q}\epsilon^{\frac{(2 - \beta)q}{2}} = c \, \epsilon^{(1 + \gamma_0 - \beta/2)q}.
\end{align}
It remains to estimate $I_{2, 3}$. By  \eqref{eq2018-11-23-1}, \eqref{eq2018-09-12-5} and \cite[(6.3)]{DKN13},
\begin{align}\label{eq2018-09-13-4}
I_{2,3} & \leq c\, (t - s - c_0\epsilon^{\gamma_0})^{q - 1}\int_{s + c_0\epsilon^{\gamma_0}}^td\rho \,\Big(\int_{s - \epsilon}^{s \wedge \rho}d\theta \int_{\mathbb{R}^k}\mu(d\xi)|\mathscr{F}S(\rho - \theta, *)(\xi)|^2\Big)^q \nonumber \\
 & = c\, (t - s - c_0\epsilon^{\gamma_0})^{q - 1} \int_{s + c_0\epsilon^{\gamma_0}}^td\rho \,\big((\rho - s + \epsilon)^{\frac{2 - \beta}{2}} -  (\rho - s)^{\frac{2 - \beta}{2}}\big)^q \nonumber \\
 &  \leq c\, (t - s - c_0\epsilon^{\gamma_0})^{q} \big((c_0\epsilon^{\gamma_0} + \epsilon)^{\frac{2 - \beta}{2}} -  (c_0\epsilon^{\gamma_0})^{\frac{2 - \beta}{2}}\big)^q \nonumber \\
 &  \leq c\, \epsilon^{\frac{2 - \beta}{2}\gamma_0q}\big((1 + c_0^{-1}\epsilon^{1 - \gamma_0})^{\frac{2 - \beta}{2}} -  1\big)^q \nonumber \\
 &  \leq c\,\epsilon^{(\frac{2 - \beta}{2}\gamma_0 + 1 - \gamma_0)q} =  c\,\epsilon^{(1 - \beta\gamma_0/2)q},
\end{align}
where in the second inequality we use Lemma \ref{lemma2018-11-22-2}(a) since $\rho - s \geq c_0\epsilon^{\gamma_0}$ for all $\rho \in [s + c_0\epsilon^{\gamma_0}, t]$,  in the third inequality we bound $t - s - c_0\epsilon^{\gamma_0}$ by $T$, and in the fourth inequality we use Lemma \ref{lemma2018-11-22-2}(b).

Finally, we combine the estimates in \eqref{eq2018-09-12-7}--\eqref{eq2018-09-13-4} to obtain \eqref{eq2018-09-12-2}.
\end{proof}

We now prove Theorem \ref{prop2017-09-19-4000}.

\begin{proof}[Proof of Theorem \ref{prop2017-09-19-4000}]
The proof of this theorem follows lines similar to those of \cite[Proposition 5.6]{DKN13}.

\textbf{Case 1}. Assume $t - s > 0$ and $\|x - y\|^2 \leq t - s$. Fix $\epsilon \in \, ]0, \delta(t - s)[$, where $0 < \delta < 1$ is fixed; its specific value will be decided on later (see the line above \eqref{eq2018-09-14-15}). For $\xi = (\lambda, \mu) \in \mathbb{R}^{2d}$ with $\|\xi\|^2 = \|\lambda\|^2 + \|\mu\|^2 = 1$,  we write
\begin{align*}
\xi^T\gamma_Z\xi \geq J_1 + J_2,
\end{align*}
where
\begin{align}
J_1 &:= \int_{s - \epsilon}^sdr\sum_{l = 1}^d\Big\|\sum_{i = 1}^d(\lambda_i - \mu_i)[S(s - r, y - *)\sigma_{il}(u(r, *)) + a_i(l, r, s, y)] + W\Big\|_{\mathscr{H}}^2, \label{eq2018-09-17-2}\\
J_2 &:= \int_{t - \epsilon}^tdr\sum_{l = 1}^d \|W\|_{\mathscr{H}}^2,\label{eq2018-09-17-3}
\end{align}
where
\begin{align}
W := \sum_{i = 1}^d[\mu_i S(t - r, x - *)\sigma_{il}(u(r, *)) + \mu_ia_i(l, r, t, x)], \label{eq2018-09-17-1}
\end{align}
 and $a_i(l, r, t, x)$ is defined in \eqref{eq2018-09-12-1}.

We use the inequality
\begin{align}\label{eq2018-09-13-5}
\|a + b\|^2_\mathscr{H} \geq \frac{2}{3}\|a\|^2_{\mathscr{H}} - 2\|b\|^2_{\mathscr{H}},
\end{align}
subtract and add a "local" term to find that
$
J_2 \geq \frac{2}{3}J_2^{(1)} - 4(J_2^{(2)} + J_2^{(3)}),
$
where
\begin{align*}
J_2^{(1)} &= \sum_{l = 1}^d \int_{t - \epsilon}^tdr \int_{\mathbb{R}^k}dv\int_{\mathbb{R}^k}dz \, \|v - z\|^{-\beta} \\
& \qquad \qquad \times S(t - r, x - v)S(t - r, x - z)(\mu^{T} \cdot \sigma(u(r, x)))_l^2, \\
J_2^{(2)} &=  \sum_{l = 1}^d \int_{t - \epsilon}^tdr \int_{\mathbb{R}^k}dv\int_{\mathbb{R}^k}dz \, \|v - z\|^{-\beta}S(t - r, x - v)S(t - r, x - z) \\
& \qquad \qquad \times \big(\mu^{T} \cdot [\sigma(u(r, v)) - \sigma(u(r, x))]\big)_l\big(\mu^{T} \cdot [\sigma(u(r, z)) - \sigma(u(r, x))]\big)_l, \\
J_2^{(3)} &= \int_{t - \epsilon}^tdr\sum_{l = 1}^d \Big\|\sum_{i = 1}^da_i(l, r, t, x)\mu_i\Big\|_{\mathscr{H}}^2.
\end{align*}
Now, hypothesis \textbf{P2} and \cite[Lemma 6.1]{DKN13} together imply that
\begin{align}\label{eq2018-09-13-6}
J_2^{(1)} \geq c \, \|\mu\|\epsilon^{\frac{2 - \beta}{2}}.
\end{align}
Similar to the calculation in \cite[(4.4)]{DKN13}, we can replace the exponent $\gamma$ there by $2 - \beta$ by using our Theorem \ref{prop2018-09-10-1} instead of their $(2.6)$ to obtain that, for any $q \geq 1$,
\begin{align}\label{eq2018-09-13-7}
\mbox{E}\Big[\sup_{\|\xi\| = 1}|J_2^{(2)}|^q\Big] \leq c \, \epsilon^{(2 - \beta)q}.
\end{align}
Moreover, applying \cite[Lemma 6.2]{DKN13} with $a = 1$ and $s = t$,
\begin{align}\label{eq2018-09-13-8}
\mbox{E}\Big[\sup_{\|\xi\| = 1}|J_2^{(3)}|^q\Big] \leq c \, \epsilon^{(2 - \beta)q}.
\end{align}
We will bound $J_1$ in two different subcases.

{\em Subcase A}: $\epsilon < \delta(t - s)^{1/\gamma_0}$ where $\gamma_0 \in \, ]\frac{1}{2}, 1[$.
 We use \eqref{eq2018-09-13-5} again and we subtract and add a "local" term to see that
\begin{align*}
J_1 \geq \frac{2}{3}J_1^{(1)} - 8(J_1^{(2)} + J_1^{(3)} + J_1^{(4)} + J_1^{(5)}),
\end{align*}
where
\begin{align*}
J_1^{(1)} & = \sum_{l = 1}^d \int_{s - \epsilon}^sdr \, ((\lambda - \mu)^{T} \cdot \sigma(u(r, y)))_l^2\int_{\mathbb{R}^k}d\xi \, \|\xi\|^{\beta - k}|\mathscr{F}S(s - r, y - *)(\xi)|^2, \\
J_1^{(2)} & = \sum_{l = 1}^d \int_{s - \epsilon}^sdr \int_{\mathbb{R}^k}dv\int_{\mathbb{R}^k}dz \, \|v - z\|^{-\beta} S(s - r, y - v)S(s - r, y - z)\\
& \qquad  \times \big((\lambda - \mu)^{T} \cdot [\sigma(u(r, v)) - \sigma(u(r, y))]\big)_l\big((\lambda - \mu)^{T} \cdot [\sigma(u(r, z)) - \sigma(u(r, y))]\big)_l, \\
J_1^{(3)} & =  \int_{s - \epsilon}^sdr \sum_{l = 1}^d\Big\|\sum_{i = 1}^d\mu_iS(t - r, x - *)\sigma_{il}(u(r, *))\Big\|_{\mathscr{H}}^2, \\
J_1^{(4)} & =  \int_{s - \epsilon}^sdr \sum_{l = 1}^d\Big\|\sum_{i = 1}^d(\lambda_i - \mu_i)a_i(l, r, s, y)\Big\|_{\mathscr{H}}^2, \\
J_1^{(5)} & =  \int_{s - \epsilon}^sdr \sum_{l = 1}^d\Big\|\sum_{i = 1}^d\mu_ia_i(l, r, t, x)\Big\|_{\mathscr{H}}^2.
\end{align*}

Hypothesi \textbf{P2} and \cite[Lemma 6.1]{DKN13} together imply that
\begin{align}\label{eq2018-09-13-10}
J_1^{(1)} &\geq c \, \|\lambda - \mu\|\epsilon^{\frac{2 - \beta}{2}}.
\end{align}

Similar to the term $J_2^{(2)}$ and \eqref{eq2018-09-13-7}, we obtain that, for any $q \geq 1$,
\begin{align}\label{eq2018-09-13-11}
\mbox{E}\Big[\sup_{\|\xi\| = 1}|J_1^{(2)}|^q\Big] \leq c \, \epsilon^{(2 - \beta)q}.
\end{align}

Using hypothesis \textbf{P1} and \cite[lemma 6.1]{DKN13},
\begin{align}\label{eq2018-09-13-12}
J_1^{(3)} &\leq c\, \int_{s - \epsilon}^sdr\int_{\mathbb{R}^k}d\xi \, \|\xi\|^{\beta - k}|\mathscr{F}S(t - r, x - *)(\xi)|^2\nonumber \\
& = c\, \big((t - s + \epsilon)^{\frac{2 - \beta}{2}} - (t - s)^{\frac{2 - \beta}{2}}\big)\nonumber \\
& \leq c \, \big((\delta^{-\gamma_0}\epsilon^{\gamma_0} + \epsilon)^{\frac{2 - \beta}{2}} - (\delta^{-\gamma_0}\epsilon^{\gamma_0})^{\frac{2 - \beta}{2}}\big)\nonumber \\
& = c \, \epsilon^{\frac{2 - \beta}{2}\gamma_0}\big((1 + \delta^{\gamma_0}\epsilon^{1 - \gamma_0})^{\frac{2 - \beta}{2}} - 1\big)\leq c' \, \epsilon^{1 - \beta\gamma_0/2},
\end{align}
where the second inequality holds by Lemma \ref{lemma2018-11-22-2}(a) because $t - s > \delta^{-\gamma_0}\epsilon^{\gamma_0}$, and the last inequality holds by Lemma \ref{lemma2018-11-22-2}(b).

Similar to the term $J_2^{(3)}$ and \eqref{eq2018-09-13-8}, we see that, for any $q \geq 1$,
\begin{align}\label{eq2018-09-13-13}
\mbox{E}\Big[\sup_{\|\xi\| = 1}|J_1^{(4)}|^q\Big] \leq c \, \epsilon^{(2 - \beta)q}.
\end{align}

To estimate $J_1^{(5)}$, since we are under the assumption $t - s > \delta^{-\gamma_0}\epsilon^{\gamma_0}$, by Lemma \ref{lemma2018-09-12-1}, for any $q \geq 1$,
\begin{align}\label{eq2018-09-13-14}
\mbox{E}\Big[\sup_{\|\xi\| = 1}|J_1^{(5)}|^q\Big] \leq c \, \epsilon^{\min\left\{\frac{2 - \beta}{2}(1 + \gamma_0), \, 1 - \frac{\beta\gamma_0}{2}\right\}q}.
\end{align}
From \eqref{eq2018-09-13-6}--\eqref{eq2018-09-13-14}, we conclude that in the subcase $\epsilon < \delta(t - s)^{1/\gamma_0}$,
\begin{align}\label{eq2018-09-13-15}
\inf_{\|\xi\| = 1}\xi^T\gamma_Z\xi \geq c\, \epsilon^{\frac{2 - \beta}{2}} - Z_{\epsilon}^1,
\end{align}
where $Z_{\epsilon}^1:= \sup_{\|\xi\| = 1}8(J_2^{(2)} + J_2^{(3)} + J_1^{(2)} + J_1^{(3)} + J_1^{(4)} + J_1^{(5)})$ satisfies that, for any $q \geq 1$,
\begin{align}\label{eq2018-09-13-16}
\mbox{E}\Big[\sup_{\|\xi\| = 1}|Z_{\epsilon}^1|^q\Big]\leq c \, \epsilon^{\min\left\{\frac{2 - \beta}{2}(1 + \gamma_0), \, 1 - \frac{\beta\gamma_0}{2}\right\}q}.
\end{align}

{\em Subcase B}: $\delta(t - s)^{1/\gamma_0} \leq \epsilon < \delta(t - s)$. In this subcase, we give a different estimate on $J_1$. Apply inequality \eqref{eq2018-09-13-5} and subtract and add a "local" term, to find that
\begin{align}
J_1 \geq \frac{2}{3}A_1 - 8(A_2 + A_3 + A_4 + A_5), \label{eq2018-09-17-4}
\end{align}
where
\begin{align}
A_1 &= \sum_{l = 1}^d\int_{s - \epsilon}^sdr \, \Big\|S(s - r, y - *)\big((\lambda - \mu)^{T}\cdot\sigma(u(r, y))\big)_l \nonumber\\
& \qquad \qquad \qquad \qquad + S(t - r, x - *)\big(\mu^{T}\cdot\sigma(u(r, x))\big)_l\Big\|_{\mathscr{H}}^2, \label{eq2018-09-17-5} \\
A_2 &= \sum_{l = 1}^d\int_{s - \epsilon}^sdr \,\Big\|S(s - r, y - *)\big((\lambda - \mu)^{T}\cdot[\sigma(u(r, *)) - \sigma(u(r, y))]\big)_l \Big\|_{\mathscr{H}}^2,\label{eq2018-09-17-7} \\
A_3 &= \sum_{l = 1}^d\int_{s - \epsilon}^sdr \,\Big\|S(t - r, x - *)\big(\mu^{T}\cdot[\sigma(u(r, *)) - \sigma(u(r, x))]\big)_l \Big\|_{\mathscr{H}}^2, \label{eq2018-09-17-8}\\
A_4 &= \sum_{l = 1}^d\int_{s - \epsilon}^sdr \,\Big\|\sum_{i = 1}^d(\lambda_i - \mu_i)a_i(l, r, s, y)\Big\|_{\mathscr{H}}^2, \label{eq2018-09-17-9}\\
A_5 &= \sum_{l = 1}^d\int_{s - \epsilon}^sdr \,\Big\|\sum_{i = 1}^d\mu_ia_i(l, r, t, x)\Big\|_{\mathscr{H}}^2.\label{eq2018-09-17-10}
\end{align}

Using the inequality $\|a + b\|^2_{\mathscr{H}} \geq \|a\|^2 + \|b\|^2 - 2 |\langle a, b\rangle_{\mathscr{H}}|$, we see that $A_1 \geq \tilde{A}_1 + \tilde{A}_2 - 2\tilde{B}_4$, where
\begin{align}
\tilde{A}_1 &= \sum_{l = 1}^d\int_{s - \epsilon}^sdr \,\Big\|S(s - r, y - *)\big((\lambda - \mu)^{T}\cdot\sigma(u(r, y))\big)_l \Big\|_{\mathscr{H}}^2,\label{eq2018-09-17-11} \\
\tilde{A}_2 &= \sum_{l = 1}^d\int_{s - \epsilon}^sdr\ \Big\|S(t - r, x - *)\big(\mu^{T}\cdot\sigma(u(r, x))\big)_l \Big\|_{\mathscr{H}}^2,\label{eq2018-09-17-12} \\
\tilde{B}_4 &= \sum_{l = 1}^d\int_{s - \epsilon}^sdr \,\Big|\big\langle S(s - r, y - *)\big((\lambda - \mu)^{T}\cdot\sigma(u(r, y))\big)_l, \nonumber\\
& \qquad \qquad \qquad \qquad\qquad \qquad S(t - r, x - *)\big(\mu^{T}\cdot\sigma(u(r, x))\big)_l  \big\rangle_{\mathscr{H}}\Big|.\label{eq2018-09-17-13}
\end{align}
 Hypothesis \textbf{P2} and \cite[Lemma 6.1]{DKN13} together imply that
\begin{align}\label{eq2018-09-13-60}
\tilde{A}_1 \geq c \, \|\lambda - \mu\|\epsilon^{\frac{2 - \beta}{2}}.
\end{align}
From \eqref{eq2018-09-13-6} and \eqref{eq2018-09-13-60}, since $\|\lambda\|^2 + \|\mu\|^2 = 1$, we see that
\begin{align}\label{eq2018-09-14-1}
J_1 + J_2 \geq c_0 \, \epsilon^{\frac{2 - \beta}{2}} - \frac{4}{3}\tilde{B}_4 - 8(J_2^{(2)} + J_2^{(3)} + A_2 + A_3 + A_4 + A_5).
\end{align}

We have bounded the two terms $J_2^{(2)}$ and $J_2^{(3)}$ in \eqref{eq2018-09-13-7} and \eqref{eq2018-09-13-8}. We now estimate the other five terms on the right-hand side of \eqref{eq2018-09-14-1}. As for the term $J_2^{(2)}$ and \eqref{eq2018-09-13-7}, we obtain that, for any $q \geq 1$,
\begin{align}\label{eq2018-09-13-110}
\mbox{E}\Big[\sup_{\|\xi\| = 1}|A_2|^q\Big] \leq c \, \epsilon^{(2 - \beta)q}.
\end{align}
Similar to $J_2^{(3)}$ and \eqref{eq2018-09-13-8},, we have for any $q \geq 1$,
\begin{align}\label{eq2018-09-13-1101}
\mbox{E}\Big[\sup_{\|\xi\| = 1}|A_4|^q\Big] \leq c \, \epsilon^{(2 - \beta)q}.
\end{align}
Again, by \cite[Lemma 6.2]{DKN13} with $a = 1$, for any $q \geq 1$,
\begin{align}\label{eq2018-09-13-11012}
\mbox{E}\Big[\sup_{\|\xi\| = 1}|A_5|^q\Big] &\leq c \, (t - s + \epsilon)^{\frac{2 - \beta}{2}q}\epsilon^{\frac{2 - \beta}{2}q} \nonumber \\
&\leq c \, (\delta^{-\gamma_0}\epsilon^{\gamma_0} + \epsilon)^{\frac{2 - \beta}{2}q}\epsilon^{\frac{2 - \beta}{2}q}  \leq c' \, \epsilon^{\frac{2 - \beta}{2}(1 + \gamma_0)q},
\end{align}
where, in the second inequality, we have used the assumption $t - s \leq \delta^{-\gamma_0}\epsilon^{\gamma_0}$.

We next bound the $q$-th moment of $A_3$. This is similar to the calculation in \cite[p.129]{DKN13}, but with their exponent $\gamma$ replaced by $2 - \beta$ since now we use our Theorem \ref{prop2018-09-10-1} instead of their $(2.6)$. Hence,
$
\mbox{E}[\sup_{\|\xi\| = 1}|A_3|^q] \leq c \, a_1 \times a_2,
$
where $a_1$ and $a_2$ are defined in \cite[p.129]{DKN13}, that is,
\begin{align}
a_1 &= \Big(\int_{s - \epsilon}^sdr \int_{\mathbb{R}^k}dv\int_{\mathbb{R}^k}dz \, \|v - z\|^{-\beta} S(t - r, x - v)S(t - r, x - z)\Big)^{q - 1},\label{eq2018-09-18-1} \\
a_2 &=\int_{s - \epsilon}^sdr \int_{\mathbb{R}^k}dv\int_{\mathbb{R}^k}dz \, \|v - z\|^{-\beta} S(t - r, x - v) \nonumber \\
&\qquad \qquad \qquad\qquad \qquad \times S(t - r, x - z)\|v -x\|^{\frac{2 - \beta}{2}q}\|z -x\|^{\frac{2 - \beta}{2}q}\label{eq2018-09-18-2}
\end{align}
By \cite[Lemma 6.1]{DKN13},
\begin{align}\label{eq2018-09-18-3}
a_1 &= c\, \big((t - s + \epsilon)^{\frac{2 - \beta}{2}} - (t - s)^{\frac{2 - \beta}{2}}\big)^{q - 1} \nonumber \\
& \leq c \, (t - s + \epsilon)^{\frac{2 - \beta}{2}(q - 1)}  \leq c \, (\delta^{-\gamma_0}\epsilon^{\gamma_0} + \epsilon)^{\frac{2 - \beta}{2}(q - 1)}\nonumber \\
& \leq c\, \epsilon^{\frac{2 - \beta}{2}\gamma_0(q - 1)},
\end{align}
where, in the second inequality, we use the assumption $t - s \leq \delta^{-\gamma_0}\epsilon^{\gamma_0}$.
For $a_2$, as in \cite[p.129]{DKN13}, we use the change of variables $\tilde{v} = \frac{x - v}{\sqrt{t - r}}$, $\tilde{z} = \frac{x - z}{\sqrt{t - r}}$, to see that
\begin{align}\label{eq2018-09-18-4}
a_2 &= \int_{s - \epsilon}^sdr \, (t - r)^{\frac{(2 - \beta)q}{2}- \frac{\beta}{2}} \int_{\mathbb{R}^k}d\tilde{v}\int_{\mathbb{R}^k}d\tilde{z}\, S(1, \tilde{v})S(1, \tilde{z})\|\tilde{v} - \tilde{z}\|^{-\beta}\|\tilde{v}\|^{\frac{(2 - \beta)q}{2}}\|\tilde{z}\|^{\frac{(2 - \beta)q}{2}}\nonumber \\
&= c \, \big((t - s + \epsilon)^{\frac{2 - \beta}{2}(1 + q)} - (t - s)^{\frac{2 - \beta}{2}(1 + q)}\big) \nonumber\\
 & \leq c \, (\delta^{-\gamma_0}\epsilon^{\gamma_0} + \epsilon)^{\frac{2 - \beta}{2}(1 + q)} \leq c'\, \epsilon^{\frac{2 - \beta}{2}\gamma_0(1 + q)},
\end{align}
where, in the first inequality, we use the assumption $t - s \leq \delta^{-\gamma_0}\epsilon^{\gamma_0}$. Therefore, from \eqref{eq2018-09-18-3} and \eqref{eq2018-09-18-4}, we obtain
\begin{align}\label{eq2018-09-14-10}
\mbox{E}\Big[\sup_{\|\xi\| = 1}|A_3|^q\Big] &\leq c \, \epsilon^{\frac{2 - \beta}{2}\gamma_0(q - 1) + \frac{2 - \beta}{2}\gamma_0(1 + q)} = c \, \epsilon^{(2 - \beta)\gamma_0q}.
\end{align}

We proceed to study the term $\tilde{B}_4$. Following the calculation in \cite[p.130]{DKN13}, by hypothesis \textbf{P1} and the semigroup property of $S(t, v)$,
\begin{align}\label{eq2018-09-14-11}
\tilde{B}_4 &\leq \tilde{c} \, \int_{0}^\epsilon dr \int_{\mathbb{R}^k}dv \, \|v\|^{-\beta}(t - s + 2r)^{-k/2}\exp\Big(-\frac{\|y - x + v\|^2}{2(t - s + 2r)}\Big) \nonumber \\
& \leq \tilde{c} (I_1 + I_2),
\end{align}
where
\begin{align*}
I_1 &=  \int_{0}^\epsilon dr \int_{\|v\|  \leq \theta_0\sqrt{r}}dv \, \|v\|^{-\beta}(t - s + 2r)^{-k/2}\exp\Big(-\frac{\|y - x + v\|^2}{2(t - s + 2r)}\Big), \nonumber \\
I_2 &=  \int_{0}^\epsilon dr \int_{\|v\|  > \theta_0\sqrt{r}}dv \, \|v\|^{-\beta}(t - s + 2r)^{-k/2}\exp\Big(-\frac{\|y - x + v\|^2}{2(t - s + 2r)}\Big).
\end{align*}
The constant $\theta_0$ above is a fixed and sufficiently large constant such that
\begin{align}\label{eq2018-09-14-12}
\frac{2\tilde{c}}{2 - \beta}\theta_0^{-\beta} < \frac{3c_0}{8},
\end{align}
where $c_0$ and $\tilde{c}$ are the constants in \eqref{eq2018-09-14-1} and \eqref{eq2018-09-14-11}. By the choice of $\theta_0$, we have
\begin{align}\label{eq2018-09-14-13}
\tilde{c}\, I_2 & \leq \tilde{c}\, \int_{0}^\epsilon dr\, \theta_0^{-\beta}r^{-\beta/2} \int_{\mathbb{R}^k}dv \, (t - s + 2r)^{-k/2}\exp\Big(-\frac{\|y - x + v\|^2}{2(t - s + 2r)}\Big) \nonumber \\
& = \tilde{c}\, \int_{0}^\epsilon dr\, \theta_0^{-\beta}r^{-\beta/2} = \frac{2\tilde{c}}{2 - \beta}\theta_0^{-\beta}\epsilon^{\frac{2 - \beta}{2}} \leq  \frac{3c_0}{8}\epsilon^{\frac{2 - \beta}{2}}.
\end{align}

As for $I_1$, it is bounded above by
\begin{align*}
\int_0^{\epsilon}dr \, (t - s + 2r)^{-k/2}\int_{\|v\|  \leq \theta_0\sqrt{r}}dv \, \|v\|^{-\beta},
\end{align*}
and the $dv$-integral is equal to $c_k\, \theta_0^{k - \beta}r^{(k - \beta)/2}$, so
\begin{align}
I_1 &\leq c_k\, \theta_0^{k - \beta}\int_0^{\epsilon}dr\, r^{(k - \beta)/2}(t - s + 2r)^{-k/2} \nonumber \\
&\leq c_k\, \theta_0^{k - \beta}\int_0^{\epsilon}dr\, (t - s + 2r)^{-\beta/2} \nonumber\\
& = c_k\, \theta_0^{k - \beta}(2 - \beta)^{-1} \epsilon^{\frac{2 - \beta}{2}}\big[(2 + (t - s)/\epsilon)^{\frac{2 - \beta}{2}} - ((t - s)/\epsilon)^{\frac{2 - \beta}{2}}\big] \nonumber\\
& \leq c_k\, \theta_0^{k - \beta}(2 - \beta)^{-1} \epsilon^{\frac{2 - \beta}{2}}\big[(2 + \delta^{-1})^{\frac{2 - \beta}{2}} - (\delta^{-1})^{\frac{2 - \beta}{2}}\big], \label{eq2018-09-14-16}
\end{align}
where, in the last inequality, we use the assumption $\epsilon < \delta(t - s)$ and Lemma \ref{lemma2018-11-22-2}(a). Since $\lim_{\delta \rightarrow 0} (2 + \delta^{-1})^{\frac{2 - \beta}{2}} - (\delta^{-1})^{\frac{2 - \beta}{2}} = 0$, we choose $\delta$ sufficiently small such that
\begin{align}\label{eq2018-09-14-15}
\tilde{c}\, c_k\, \theta_0^{k - \beta}(2 - \beta)^{-1} \big[(2 + \delta^{-1})^{\frac{2 - \beta}{2}} - (\delta^{-1})^{\frac{2 - \beta}{2}}\big] \leq \frac{3c_0}{16},
\end{align}
where $\tilde{c}$, $c_k$ and $c_0$ are the constants in \eqref{eq2018-09-14-11}, \eqref{eq2018-09-14-16} and \eqref{eq2018-09-14-1}.

From the estimates in \eqref{eq2018-09-13-7}, \eqref{eq2018-09-13-8}, \eqref{eq2018-09-14-1}--\eqref{eq2018-09-14-15}, we conclude that in the subcase $\delta(t - s)^{1/\gamma_0} < \epsilon < \delta(t - s)$,
\begin{align}\label{eq2018-09-14-17}
\inf_{\|\xi\| = 1}\xi^T\gamma_Z\xi \geq  \frac{c_0}{4}\epsilon^{\frac{2 - \beta}{2}} - Z_{\epsilon}^2,
\end{align}
where $Z_{\epsilon}^2:= 8\sup_{\|\xi\| = 1}(J_2^{(2)} + J_2^{(3)} + A_2 + A_3 + A_4 + A_5)$ satisfies that, for any $q \geq 1$,
\begin{align}\label{eq2018-09-13-160}
\mbox{E}\Big[|Z_{\epsilon}^2|^q\Big]\leq c \, \epsilon^{(2 - \beta)\gamma_0q}.
\end{align}

Therefore, in the \textbf{Case 1}, from \eqref{eq2018-09-13-15} and \eqref{eq2018-09-14-17}, for $\epsilon < \delta(t - s)$,
\begin{align}\label{eq2018-09-14-1700}
\inf_{\|\xi\| = 1}\xi^T\gamma_Z\xi \geq  c\, \epsilon^{\frac{2 - \beta}{2}} - Z_{\epsilon},
\end{align}
where, by \eqref{eq2018-09-13-16} and \eqref{eq2018-09-13-160}, $Z_{\epsilon}:= Z^1_{\epsilon}\, 1_{\{0 < \epsilon < \delta(t - s)^{1/\gamma_0}\}} + Z^2_{\epsilon}\, 1_{\{\delta(t - s)^{1/\gamma_0} \leq \epsilon < \delta(t - s)\}}$ satisfies that, for any $q \geq 1$,
\begin{align}\label{eq2018-09-14-18}
\mbox{E}\Big[|Z_{\epsilon}|^q\Big]\leq c \, \epsilon^{\min\left\{(2 - \beta)\gamma_0, \, 1 - \frac{\beta\gamma_0}{2}\right\}q}.
\end{align}

Since $\gamma_0 \in \, ]\frac{1}{2}, 1[$, it is clear that $\min\left\{(2 - \beta)\gamma_0, \, 1 - \beta\gamma_0/2\right\} > \frac{2 - \beta}{2}$. Therefore, we apply \cite[Proposition 3.5]{DKN09} to find that
\begin{align} \label{eq2018-09-14-19}
\mbox{E}\Big[\Big(\inf_{\|\xi\| = 1}\xi^T\gamma_Z\xi\Big)^{-2pd}\Big] &\leq c\, \left(\delta(t - s)\right)^{-2pd\frac{2 - \beta}{2}} = c'(t - s)^{-2pd\frac{2 - \beta}{2}}\nonumber \\
  &\leq \tilde{c}\left[|t - s|^{\frac{2 - \beta}{2}} + \|x - y\|^{2 - \beta}\right]^{-2pd},
\end{align}
whence follows the result in the case that $\|x - y\|^{2} \leq t - s < 1$.

\textbf{Case 2}. Now we work on the second case where $\|x - y\| > 0$ and $\|x - y\|^{2} \geq t - s \geq 0$.  Let $\epsilon > 0$ be such that $(1 + \alpha)\epsilon^{1/2} < \frac{1}{2}\|x - y\|$, where $\alpha > 0$ is large but fixed; its specific value will be decided on later (see the explanation for \eqref{eq2018-09-17-22} and \eqref{eq2018-09-17-23}).

From here on, \textbf{Case 2} is divided into two further subcases.

{\em Subcase A}. Suppose, in addition, that $\epsilon \geq \delta(t - s)$, where $\delta$ is chosen as in \textbf{Case 1}. In this subcase, we find that
\begin{align*}
\xi^T\gamma_Z\xi \geq J_1 + \tilde{J}_2,
\end{align*}
where
\begin{align*}
\tilde{J}_2 &:= \int_{(t - \epsilon)\vee s}^tdr\sum_{l = 1}^d \|W\|_{\mathscr{H}}^2
\end{align*}
with $W$ as defined in \eqref{eq2018-09-17-1}, and $J_1$ has the same expression as in \eqref{eq2018-09-17-2}.

We estimate $J_1$ in the same way as in \eqref{eq2018-09-17-4}, i.e.,
\begin{align*}
J_1 \geq \frac{2}{3}A_1 - 8(A_2 + A_3 + A_4 + A_5),
\end{align*}
where $A_1$--$A_5$ have the same expression as \eqref{eq2018-09-17-5}--\eqref{eq2018-09-17-10}. Moreover, as in Subcase A of \textbf{Case 1}, we see that $A_1 \geq \tilde{A}_1 + \tilde{A}_2 - 2\tilde{B}_4$, where $\tilde{A}_1$, $\tilde{A}_2$ and $\tilde{B}_4$ have the expressions as in \eqref{eq2018-09-17-11}--\eqref{eq2018-09-17-13}.

As for $\tilde{J}_2$, we apply \eqref{eq2018-09-13-5} and we subtract and add a "local" term to see that
\begin{align*}
\tilde{J}_2 \geq \frac{2}{3}B_1 - 4(B_2 + B_3),
\end{align*}
where
\begin{align*}
B_1 &= \sum_{l = 1}^d \int_{(t - \epsilon)\vee s}^tdr \, \|S(t - r, x - *)(\mu^{T} \cdot \sigma(u(r, x)))_l\|^2_{\mathscr{H}}, \\
B_2 &=  \sum_{l = 1}^d \int_{t - \epsilon}^tdr \, \|S(t - r, x - *)\big(\mu^{T} \cdot [\sigma(u(r, *)) - \sigma(u(r, x))]\big)_l\|_{\mathscr{H}}^2, \\
B_3 &=  \int_{t - \epsilon}^tdr\sum_{l = 1}^d \Big\|\sum_{i = 1}^da_i(l, r, t, x)\mu_i\Big\|_{\mathscr{H}}^2.
\end{align*}

By hypothesis \textbf{P2} and \cite[Lemma 6.1]{DKN13}
\begin{align}\label{eq2018-09-17-6}
\tilde{A}_2 &\geq c\, \|\mu\|^2\int_{s - \epsilon}^sdr \, \|S(t - r, x - *)\|_{\mathscr{H}}^2\nonumber\\
 &= c \, \|\mu\|^2\big((t - s + \epsilon)^{\frac{2 - \beta}{2}} - (t - s)^{\frac{2 - \beta}{2}}\big).
\end{align}
Similarly,
\begin{align}\label{eq2018-09-17-60}
B_1 &\geq c\, \|\mu\|^2\int_{(t - \epsilon) \vee s}^tdr \, \|S(t - r, x - *)\|_{\mathscr{H}}^2\nonumber\\
 &= c \, \|\mu\|^2(t - ((t - \epsilon)\vee s))^{\frac{2 - \beta}{2}} =  c \,\|\mu\|^2((t - s)\wedge \epsilon)^{\frac{2 - \beta}{2}}.
\end{align}
From \eqref{eq2018-09-13-60}, \eqref{eq2018-09-17-6} and \eqref{eq2018-09-17-60}, we see that
\begin{align}
\tilde{A}_1 + \tilde{A}_2  + B_1 &\geq c \, \Big(\|\lambda - \mu\|\epsilon^{\frac{2 - \beta}{2}} + \|\mu\|^2\big((t - s - \epsilon)^{\frac{2 - \beta}{2}} - (t - s)^{\frac{2 - \beta}{2}} + ((t - s)\wedge \epsilon)^{\frac{2 - \beta}{2}} \big)\Big) \nonumber \\
& = c \, \epsilon^{\frac{2 - \beta}{2}}\Big(\|\lambda - \mu\| + \|\mu\|^2\big((\frac{t - s}{\epsilon} + 1)^{\frac{2 - \beta}{2}} - (\frac{t - s}{\epsilon})^{\frac{2 - \beta}{2}} + ((\frac{t - s}{\epsilon})\wedge 1)^{\frac{2 - \beta}{2}} \big)\Big) \nonumber
\end{align}
Denote $\zeta(x):= (x + 1)^{\frac{2 - \beta}{2}} - x^{\frac{2 - \beta}{2}} + (x\wedge 1)^{\frac{2 - \beta}{2}}$, $x \in [0, \infty[$. Then it is clear that
$
\hat{c}_0 := \min_{0 \leq x < \infty}\zeta(x) > 0.
$
Thus we have
\begin{align}\label{eq2018-09-17-20}
\tilde{A}_1 + \tilde{A}_2  + B_1 &\geq c\, \epsilon^{\frac{2 - \beta}{2}}\Big(\|\lambda - \mu\|^{2} + \hat{c}_0\|\mu\|^2\Big) \geq c'\epsilon^{\frac{2 - \beta}{2}}.
\end{align}

The estimates of $A_2$ and $A_4$ in this subcase are the same as those in \eqref{eq2018-09-13-110} and \eqref{eq2018-09-13-1101} respectively. Hence we have
for any $q \geq 1$,
\begin{align}\label{eq2018-09-13-1102}
\mbox{E}\Big[\sup_{\|\xi\| = 1}|A_2|^q\Big] \leq c \, \epsilon^{(2 - \beta)q}, \quad\quad  \mbox{E}\Big[\sup_{\|\xi\| = 1}|A_4|^q\Big] \leq c \, \epsilon^{(2 - \beta)q}.
\end{align}
The term $B_2$ is similar to $J_2^{(2)}$ and $B_3$ is similar to $J_2^{(3)}$, so as in \eqref{eq2018-09-13-7} and \eqref{eq2018-09-13-8}, for any $q \geq 1$,
\begin{align}\label{eq2018-09-13-1109}
\mbox{E}\Big[\sup_{\|\xi\| = 1}|B_2|^q\Big] \leq c \, \epsilon^{(2 - \beta)q}, \quad\quad  \mbox{E}\Big[\sup_{\|\xi\| = 1}|B_3|^q\Big] \leq c \, \epsilon^{(2 - \beta)q}.
\end{align}
Moreover, as in \eqref{eq2018-09-13-11012}, by \cite[Lemma 6.2]{DKN13} with $a = 1$, for any $q \geq 1$,
\begin{align}\label{eq2018-09-17-11012}
\mbox{E}\Big[\sup_{\|\xi\| = 1}|A_5|^q\Big] &\leq c \, (t - s + \epsilon)^{\frac{2 - \beta}{2}q}\epsilon^{\frac{2 - \beta}{2}q} \nonumber \\
&\leq c \, (\delta^{-1}\epsilon + \epsilon)^{\frac{2 - \beta}{2}q}\epsilon^{\frac{2 - \beta}{2}q} =  c' \, \epsilon^{(2 - \beta)q}.
\end{align}
where, in the second inequality, we have used the assumption $t - s \leq \delta^{-1}\epsilon$.

We proceed to estimate the $q$-th moment of $A_3$. Again, we follow the calculation in \cite[p.129]{DKN13}, use our Theorem \ref{prop2018-09-10-1} instead of their $(2.6)$, to replace their exponent $\gamma$ by $2 - \beta$. That is,
$
\mbox{E}[\sup_{\|\xi\| = 1}|A_3|^q] \leq c \, a_1 \times a_2,
$
where $a_1$ and $a_2$ are defined in \cite[p.129]{DKN13} (see also our \eqref{eq2018-09-18-1} and \eqref{eq2018-09-18-2}). Similar to \eqref{eq2018-09-18-3}, we have, in this subcase $t - s \leq \delta^{-1}\epsilon$,
\begin{align}
a_1 &= c\, \big((t - s + \epsilon)^{\frac{2 - \beta}{2}} - (t - s)^{\frac{2 - \beta}{2}}\big)^{q - 1} \nonumber \\
& \leq c \, (t - s + \epsilon)^{\frac{2 - \beta}{2}(q - 1)}  \leq c \, (\delta^{-1}\epsilon + \epsilon)^{\frac{2 - \beta}{2}(q - 1)} =  c'\, \epsilon^{\frac{2 - \beta}{2}(q - 1)}.
\end{align}
Moreover, similar to \eqref{eq2018-09-18-4},
\begin{align}
a_2  &= c \, \big((t - s + \epsilon)^{\frac{2 - \beta}{2}(1 + q)} - (t - s)^{\frac{2 - \beta}{2}(1 + q)}\big) \nonumber\\
 & \leq c \, (\delta^{-1}\epsilon + \epsilon)^{\frac{2 - \beta}{2}(1 + q)} = c'\, \epsilon^{\frac{2 - \beta}{2}(1 + q)},
\end{align}
where, in the  inequality, we use the assumption $t - s \leq \delta^{-1}\epsilon$. Therefore,
\begin{align}\label{eq2018-09-14-101122}
\mbox{E}\Big[\sup_{\|\xi\| = 1}|A_3|^q\Big] &\leq c \, \epsilon^{\frac{2 - \beta}{2}(q - 1) + \frac{2 - \beta}{2}(1 + q)} =  c \, \epsilon^{(2 - \beta)q}.
\end{align}

Furthermore, under the assumption $(1 + \alpha)\epsilon^{1/2} < \frac{1}{2}\|x - y\|$, the estimate of $\tilde{B}_4$ follows exactly the same lines as in \cite[p.130-131]{DKN13}. Indeed, by hypothesis \textbf{P1} and using change of variable,
\begin{align*}
\tilde{B}_4 \leq c(I_1 + I_2),
\end{align*}
where
\begin{align*}
I_1 &=  \int_{0}^\epsilon dr \int_{\|v\|  < \sqrt{r}(1 + \alpha)}dv \, \|v\|^{-\beta}(t - s + 2r)^{-k/2}\exp\Big(-\frac{\|y - x + v\|^2}{2(t - s + 2r)}\Big), \nonumber \\
I_2 &=  \int_{0}^\epsilon dr \int_{\|v\|  \geq \sqrt{r}(1 + \alpha)}dv \, \|v\|^{-\beta}(t - s + 2r)^{-k/2}\exp\Big(-\frac{\|y - x + v\|^2}{2(t - s + 2r)}\Big).
\end{align*}
Using the change of variables $\rho = \frac{t - s + 2r}{\alpha^2\epsilon}$ and the inequality $t - s \leq \delta^{-1}\epsilon$,
we follow the calculation in \cite[p.130-131]{DKN13} to see that, under the assumption $(1 + \alpha)\epsilon^{1/2} < \frac{1}{2}\|x - y\|$,
\begin{align*}
I_1 \leq  \epsilon^{\frac{2 - \beta}{2}}\Phi_1(\alpha),
\end{align*}
where $\lim_{\alpha \rightarrow \infty}\Phi_1(\alpha) = 0$. The estimate of $I_2$ is the same as that in \cite[p.131]{DKN13} and we have
\begin{align*}
I_2 \leq c(1 + \alpha)^{-\beta}\epsilon^{\frac{2 - \beta}{2}}.
\end{align*}
We note that $\lim_{\alpha \rightarrow \infty}(1 + \alpha)^{-\beta} = 0$ and so we have shown that
\begin{align}\label{eq2018-09-17-22}
\tilde{B}_4 \leq \Phi(\alpha)\epsilon^{\frac{2 - \beta}{2}}, \quad \mbox{with} \quad \lim_{\alpha \rightarrow \infty} \Phi(\alpha) = 0.
\end{align}

Therefore, from \eqref{eq2018-09-17-20} and \eqref{eq2018-09-17-22}, we have shown that, in the subcase $\epsilon \geq \delta(t - s)$,
\begin{align*}
\inf_{\|\xi\| = 1}\xi^T\gamma_Z\xi \geq  c'\, \epsilon^{\frac{2 - \beta}{2}} - \frac{4}{3}  \Phi(\alpha)\epsilon^{\frac{2 - \beta}{2}} - \tilde{Z}_{\epsilon},
\end{align*}
where by \eqref{eq2018-09-13-1102}--\eqref{eq2018-09-14-101122},  $\tilde{Z}_{\epsilon}:= \sup_{\|\xi\| = 1}8(A_2 + A_3 +A_4 +A_5 + B_2 +B_3)$ satisfies that, for any $q \geq 1$,
\begin{align}\label{eq2018-09-17-24}
\mbox{E}\Big[|\tilde{Z}_{\epsilon}|^q\Big] \leq c \, \epsilon^{(2 - \beta)q}.
\end{align}
We choose $\alpha$ large enough so that $\frac{4}{3}  \Phi(\alpha) < \frac{c'}{2}$, to get
\begin{align}\label{eq2018-09-17-23}
\inf_{\|\xi\| = 1}\xi^T\gamma_Z\xi \geq  \frac{c'}{2}\, \epsilon^{\frac{2 - \beta}{2}} - \tilde{Z}_{\epsilon}.
\end{align}

{\em Subcase B}. In this final subcase, we suppose that $\epsilon < \delta(t - s) \leq \delta\|x - y\|^{2}$. Choose and fix $0 < \epsilon < \delta(t - s)$. During the course of our proof of \textbf{Case 1}, we established that
\begin{align}\label{eq2017-09-20-8}
\inf_{\|\xi\| = 1}\xi^T\gamma_Z\xi &\geq c\epsilon^{\frac{2 - \beta}{2}} - Z_{\epsilon},
\end{align}
where, for all $q \geq 1$,
\begin{align*}
\mbox{E}\left[|Z_{\epsilon}|^q\right] \leq c \, \epsilon^{\min\left\{(2 - \beta)\gamma_0, \, 1 - \frac{\beta\gamma_0}{2}\right\}q},
\end{align*}
with $\gamma_0 \in \, ]1/2, 1[$ be a fixed constant (see \eqref{eq2018-09-14-1700} and \eqref{eq2018-09-14-18}). This inequality remains valid in this Subcase B.

Combine Subcases A and B, and, in particular, \eqref{eq2018-09-17-23} and \eqref{eq2017-09-20-8}, to find that for all $0 < \epsilon < 2^{-2}(1 + \alpha)^{-2}\|x - y\|^{2}$,
\begin{align*}
\inf_{\|\xi\| = 1}\xi^T\gamma_Z\xi &\geq c \, \epsilon^{\frac{2 - \beta}{2}} -(Z_{\epsilon}\, 1_{\{\epsilon < \delta(t - s)\}} + \tilde{Z}_{\epsilon}\, 1_{\{t - s \leq \delta^{-1}\epsilon\}}).
\end{align*}
Because of this and \eqref{eq2018-09-14-18} and \eqref{eq2018-09-17-24}, by \cite[Proposition 3.5]{DKN09}, this implies that
\begin{align*}
\mbox{E}\left[\left(\inf_{ \|\xi\| = 1}\xi^T\gamma_Z\xi\right)^{-2pd}\right] &\leq c\, \|x - y\|^{2(-2dp)(\frac{2 - \beta}{2})}\\
&\leq c\left(\|x - y\|^{2} + |t - s|\right)^{(\frac{2 - \beta}{2})(-2dp)}\\
&\leq c\left(|t - s|^{\frac{2 - \beta}{2}} + \|x - y\|^{2 - \beta} \right)^{-2dp}.
\end{align*}
This completes the proof of Theorem \ref{prop2017-09-19-4000}.
\end{proof}

\begin{appendices}
\section{Proof of \eqref{eq2018-09-12-5}}
\begin{proof}[Proof of \eqref{eq2018-09-12-5}]
By \eqref{eq2018-03-01-4},
\begin{align}\label{eq2018-11-22-2}
&\mathrm{E}\left[\|D^{(l)}_{\cdot, *}(u_i(\rho, x))\|_{\mathscr{H}_{s - \epsilon, s}}^{2q}\right] = \mathrm{E}\left[\|D^{(l)}_{\cdot, *}(u_i(\rho, x))\|_{\mathscr{H}_{s - \epsilon, s\wedge \rho}}^{2q}\right] \nonumber \\
&\quad \leq c\,\mathrm{E}\bigg[\Big(\int_{s - \epsilon}^{s\wedge \rho}dr\, \|\sigma_{il}(u(r, *))S(\rho - r, x - *)\|_{\mathscr{H}}^2\Big)^q\bigg] \nonumber \\
& \quad \quad + c\,\mathrm{E}\bigg[\Big\| \int_{s - \epsilon}^{\rho}\int_{\mathbb{R}^k} S(\rho - \theta, x - \eta)1_{\{\theta > \cdot\}}\sum_{j = 1}^dD_{\cdot, *}^{(l)}(\sigma_{ij}(u(\theta, \eta)))M^j(d\theta, d\eta)\Big\|_{\mathscr{H}_{s - \epsilon, s\wedge \rho}}^{2q}\bigg] \nonumber \\
& \quad \quad + c\,\mathrm{E}\bigg[\Big\| \int_{s - \epsilon}^{\rho}\int_{\mathbb{R}^k} S(\rho - \theta, x - \eta)1_{\{\theta > \cdot\}}D_{\cdot, *}^{(l)}(b_{i}(u(\theta, \eta)))d\theta d\eta\Big\|_{\mathscr{H}_{s - \epsilon, s\wedge \rho}}^{2q}\bigg] \nonumber \\
& \quad =: A + B + C.
\end{align}

By hypothesis \textbf{P1}, it is clear that
\begin{align}\label{eq2018-11-22-3}
A \leq c \, \Big(\int_{s - \epsilon}^{s \wedge \rho}d\theta \int_{\mathbb{R}^k}\mu(d\xi)|\mathscr{F}S(\rho - \theta, *)(\xi)|^2\Big)^q.
\end{align}
Using Burkholder's inequality for Hilbert-space-valued martingales (\cite[E.2. p.212]{MeM82}) and hypothesis \textbf{P1}, we see that
\begin{align}\label{eq2018-11-22-4}
B &\leq c \, \sum_{j = 1}^d\mathrm{E}\bigg[\Big(\int_{s - \epsilon}^{\rho}d\theta \int_{\mathbb{R}^k}dz\int_{\mathbb{R}^k}dv \, \|z - v\|^{-\beta}S(\rho - \theta, x - z)\|D_{\cdot, *}u_j(\theta, z)\|_{\mathscr{H}^d_{s - \epsilon, s\wedge \theta}}\nonumber \\
&\qquad \qquad\qquad \qquad \times S(\rho - \theta, x - v)\|D_{\cdot, *}u_j(\theta, v)\|_{\mathscr{H}^d_{s - \epsilon, s\wedge \theta}}\Big)^q\bigg].
\end{align}
Moreover, using hypothesis \textbf{P1},
\begin{align}\label{eq2018-11-22-5}
C &\leq c \, \sum_{j = 1}^d\mathrm{E}\bigg[\Big(\int_{s - \epsilon}^{\rho}d\theta \int_{\mathbb{R}^k}d\eta \, S(\rho - \theta, x - \eta)\|D_{\cdot, *}u_j(\theta, \eta)\|_{\mathscr{H}^d_{s - \epsilon, s\wedge \theta}}\Big)^{2q}\bigg].
\end{align}

\textbf{Case 1}: $\rho \leq s$. Applying Minkowski's inequality with respect to the measure $\|z - v\|^{-\beta}S(\rho - \theta, x - z) S(\rho - \theta, x - v)dvdzd\theta$,
\begin{align}\label{eq2018-11-22-6}
B &\leq c \, \sum_{j = 1}^d\sup_{(t, x) \in [0, T] \times \mathbb{R}^k}\mbox{E}\left[\|D(u_j(t, x))\|^{2q}_{\mathscr{H}_T^d}\right] \nonumber \\
& \quad\quad \times \Big(\int_{s - \epsilon}^{\rho}d\theta \int_{\mathbb{R}^k}dz\int_{\mathbb{R}^k}dv \, \|z - v\|^{-\beta}S(\rho - \theta, x - z)S(\rho - \theta, x - v)\Big)^q\nonumber \\
& \leq c \, \Big(\int_{s - \epsilon}^{\rho}d\theta \int_{\mathbb{R}^k}\mu(d\xi)|\mathscr{F}S(\rho - \theta, *)(\xi)|^2\Big)^q,
\end{align}
where the second inequality holds by \eqref{eq2018-09-10-2}. Similarly, using Minkowski's inequality with respect to the measure $S(\rho - \theta, x - \eta)d\theta d\eta$,
\begin{align}
C &\leq c \, \sum_{j = 1}^d\sup_{(t, x) \in [0, T] \times \mathbb{R}^k}\mbox{E}\left[\|D(u_j(t, x))\|^{2q}_{\mathscr{H}_T^d}\right] \Big(\int_{s - \epsilon}^{\rho}d\theta \int_{\mathbb{R}^k}d\eta\, S(\rho - \theta, x - \eta)\Big)^{2q} \nonumber \\
& \leq c \, (\rho - s + \epsilon)^{2q} \leq c' \, (\rho - s + \epsilon)^{(2 - \beta)q/2} \nonumber  \\
& = c'' \, \Big(\int_{s - \epsilon}^{\rho}d\theta \int_{\mathbb{R}^k}\mu(d\xi)|\mathscr{F}S(\rho - \theta, *)(\xi)|^2\Big)^q,\label{eq2018-11-22-7}
\end{align}
where the second inequality holds by \eqref{eq2018-09-10-2} and the equality follows from \cite[Lemma 6.1]{DKN13}.

Hence, in \textbf{Case 1}, \eqref{eq2018-11-22-3}, \eqref{eq2018-11-22-6} and \eqref{eq2018-11-22-7} prove \eqref{eq2018-09-12-5}.
\vspace{0.2cm}

\textbf{Case 2}: $\rho > s$. Let $a_1 = s - \epsilon$, $b_1 = s$, $a_2 = s$, $b_2 = \rho$.

From \eqref{eq2018-11-22-4}, we bound $B$  by $B_1 + B_2$, where for $i = 1, 2$,
\begin{align*}
B_i &\leq c \, \sum_{j = 1}^d\mathrm{E}\bigg[\Big(\int_{a_i}^{b_i}d\theta \int_{\mathbb{R}^k}dz\int_{\mathbb{R}^k}dv \, \|z - v\|^{-\beta}S(\rho - \theta, x - z)\|D_{\cdot, *}u_j(\theta, z)\|_{\mathscr{H}^d_{s - \epsilon, s\wedge \theta}}\nonumber \\
&\qquad \qquad\qquad \qquad \times S(\rho - \theta, x - v)\|D_{\cdot, *}u_j(\theta, v)\|_{\mathscr{H}^d_{s - \epsilon, s\wedge \theta}}\Big)^q\bigg].
\end{align*}
From the calculation in \eqref{eq2018-11-22-6}, we have
\begin{align}\label{eq2018-11-22-9}
B_1 &\leq c \, \Big(\int_{s - \epsilon}^{s}d\theta \int_{\mathbb{R}^k}\mu(d\xi)|\mathscr{F}S(\rho - \theta, *)(\xi)|^2\Big)^q.
\end{align}
Applying H\"{o}lder's inequality with respect to the measure $\|z - v\|^{-\beta}S(\rho - \theta, x - z) S(\rho - \theta, x - v)dvdzd\theta$ and the Cauchy-Schwartz inequality,
\begin{align}\label{eq2018-11-22-10}
B_2 &\leq c \, \sum_{j = 1}^d\int_{s}^{\rho}d\theta \int_{\mathbb{R}^k}dz\int_{\mathbb{R}^k}dv \, \|z - v\|^{-\beta}S(\rho - \theta, x - z)S(\rho - \theta, x - v)\nonumber \\
&\qquad \qquad\qquad \qquad \times \sup_{x \in \mathbb{R}^k}\mathrm{E}\big[\|D_{\cdot, *}u_j(\theta, x)\|_{\mathscr{H}^d_{s - \epsilon, s\wedge \theta}}^{2q}\big] \nonumber \\
& = c \, \sum_{j = 1}^d\int_{s}^{\rho}d\theta\, (\rho - \theta)^{-\frac{\beta}{2}}\sup_{x \in \mathbb{R}^k}\mathrm{E}\big[\|D_{\cdot, *}u_j(\theta, x)\|_{\mathscr{H}^d_{s - \epsilon, s}}^{2q}\big],
\end{align}
where the equality follows from \cite[(6.3)]{DKN13}.

Moreover, from \eqref{eq2018-11-22-5}, we bound $C$ by $C_1 + C_2$, where for $i = 1, 2$,
\begin{align*}
C_i &\leq c \, \sum_{j = 1}^d\mathrm{E}\bigg[\Big(\int_{a_i}^{b_i}d\theta \int_{\mathbb{R}^k}d\eta \, S(\rho - \theta, x - \eta)\|D_{\cdot, *}u_j(\theta, \eta)\|_{\mathscr{H}^d_{s - \epsilon, s\wedge \theta}}\Big)^{2q}\bigg].
\end{align*}
Similar to the derivation of the first inequality in \eqref{eq2018-11-22-7}, we see that
\begin{align}\label{eq2018-11-22-12}
C_1 &\leq c\, \Big(\int_{s - \epsilon}^{s}d\theta \int_{\mathbb{R}^k}d\eta\, S(\rho - \theta, x - \eta)\Big)^{2q} = c\, \epsilon^{2q} \nonumber \\
&\leq c'\, \big((\rho - s + \epsilon)^{\frac{2 - \beta}{2}} - (\rho - s)^{\frac{2 - \beta}{2}}\big)^q \nonumber \\
 & = c'' \, \Big(\int_{s - \epsilon}^{s}d\theta \int_{\mathbb{R}^k}\mu(d\xi)|\mathscr{F}S(\rho - \theta, *)(\xi)|^2\Big)^q,
\end{align}
where the second inequality holds since by Lemma \ref{lemma2018-11-22-2}(a), $(\rho - s + \epsilon)^{\frac{2 - \beta}{2}} - (\rho - s)^{\frac{2 - \beta}{2}} \geq (T + \epsilon)^{\frac{2 - \beta}{2}} - T^{\frac{2 - \beta}{2}} \geq C \, \epsilon$ for all $\epsilon \in [0, T]$, and the second equality follows from \cite[Lemma 6.1]{DKN13}.
Applying H\"{o}lder's inequality with respect to the measure $S(\rho - \theta, x - \eta)d\theta d\theta$ and the Cauchy-Schwartz inequality,
\begin{align}\label{eq2018-11-22-10}
C_2 &\leq c \, \sum_{j = 1}^d\int_{s}^{\rho}d\theta \int_{\mathbb{R}^k}d\eta \, S(\rho - \theta, x - \eta)\sup_{x \in \mathbb{R}^k}\mathrm{E}\big[\|D_{\cdot, *}u_j(\theta, x)\|_{\mathscr{H}^d_{s - \epsilon, s\wedge \theta}}^{2q}\big] \nonumber \\
& = c \, \int_{s}^{\rho}d\theta\, \sum_{j = 1}^d\sup_{x \in \mathbb{R}^k}\mathrm{E}\big[\|D_{\cdot, *}u_j(\theta, x)\|_{\mathscr{H}^d_{s - \epsilon, s}}^{2q}\big].
\end{align}

Denote
\begin{align*}
\varphi(\theta) = \sum_{j = 1}^d\sup_{x \in \mathbb{R}^k}\mathrm{E}\big[\|D_{\cdot, *}u_j(\theta, x)\|_{\mathscr{H}^d_{s - \epsilon, s}}^{2q}\big].
\end{align*}
We deduce, from \eqref{eq2018-11-22-2}, \eqref{eq2018-11-22-3} and \eqref{eq2018-11-22-9}--\eqref{eq2018-11-22-10}, that
\begin{align*}
\varphi(\rho) & \leq c \, \Big(\int_{s - \epsilon}^{s}d\theta \int_{\mathbb{R}^k}\mu(d\xi)|\mathscr{F}S(\rho - \theta, *)(\xi)|^2\Big)^q + c\, \int_{s}^{\rho} (1 + (\rho - \theta)^{-\beta/2})\varphi(\theta)d\theta.
\end{align*}
By Gronwall's lemma (\cite[Lemma 15]{Dal99}), we conclude that for $\rho \geq s$,
\begin{align*}
\varphi(\rho) & \leq c \, \Big(\int_{s - \epsilon}^{s}d\theta \int_{\mathbb{R}^k}\mu(d\xi)|\mathscr{F}S(\rho - \theta, *)(\xi)|^2\Big)^q,
\end{align*}
which implies \eqref{eq2018-09-12-5}.
\end{proof}

\end{appendices}

\begin{small}

\vspace{1.5cm}

\noindent\textbf{Robert C. Dalang} and \textbf{Fei Pu.}
Institut de Math\'ematiques, Ecole Polytechnique
F\'ed\'erale de Lausanne, Station 8, CH-1015 Lausanne,
Switzerland.\\
Emails: \texttt{robert.dalang@epfl.ch} and \texttt{fei.pu@epfl.ch}\\
\end{small}

\begin{thebibliography}{99}





\bibitem{CJKS13}Conus, D., Joseph, M., Khoshnevisan, D. and Shiu, S.: On the chaotic character of the stochastic heat equation, II. \textit{Probab. Theory Related Fields} \textbf{156}, 483--533 (2013)



\bibitem{Dal99} Dalang, R.C.: Extending the martingale measure stochastic integral with application to spatially homogeneous S.P.D.E's. \textit{Electon. J. Probab.} \textbf{4}, 1--29 (1999)

\bibitem{DaF98} Dalang, R.C. and  Frangos, N. E.: The stochastic wave equation in two spatial dimensions. \textit{Ann. Probab.} \textbf{26}, 187--212 (1998)


\bibitem{DKN07} Dalang, R.C., Khoshnevisan, D. and Nualart, E.: Hitting probabilities for systems of non-linear stochastic heat equations with additive noise. \textit{ALEA} \textbf{3}, 231--271 (2007)

\bibitem{DKN09} Dalang, R.C., Khoshnevisan, D. and Nualart, E.: Hitting probabilities for systems of non-linear stochastic heat equations with multiplicative noise. \textit{Probab. Th. Rel. Fields} \textbf{144}, 371--424 (2009)

\bibitem{DKN13} Dalang, R.C., Khoshnevisan, D. and Nualart, E.: Hitting probabilities for systems of non-linear stochastic heat equations in spatial dimension $k \geq 1$. \textit{Stoch PDE: Anal Comp} \textbf{1}, 94--151 (2013)



\bibitem{DaP18a} Dalang, R.C. and Pu, F.: Optimal lower bounds on hitting probabilities for non-linear systems of stochastic fractional heat equations. arXiv:1810.05386 (2018)




\bibitem{HNS13} Hu, Y., Nualart, D. and Song, J.: A nonlinear stochastic heat equation: Hölder continuity and smoothness of the density of the solution.  \textit{Stochastic Process. Appl.} \textbf{123}, 1083--1103 (2003)


\bibitem{Lik17} Li, K.: H\"{o}lder continuity for stochastic fractional heat equation with colored noise. \textit{Statist. Probab. Lett.} \textbf{129}, 34--41 (2017)



\bibitem{MeM82} M\'{e}tivier, M.: Semimartingales. de Gruyter (1982)


\bibitem{Nua06} Nualart, D.: The Malliavin Calculus and Related Topics (2nd ed). Springer, London (2006)

\bibitem{Nua98} Nualart, D.: Analysis on Wiener space and anticipating stochastic calculus. Ecole d'Et\'{e} de Probabilit\'{e}s de Saint-Flour XXV. Lect. Notes in Math., vol. 1690, pp. 123--227. Springer, Heidelberg (1998)



\bibitem{NuQ07} Nualart, D. and Quer-Sardanyons, L.: Existence and smoothness of the density for spatially homogeneous SPDEs. \textit{Potential Anal.} \textbf{27}, 281--299 (2007)


\bibitem{ReY99} Revuz, D. and Yor, M.: Continuous martingales and Brownian motion (3rd ed). Springer Verlag (1999)

\bibitem{Sai98} Saint Loubert Bi\'{e}, E.: \'{E}tude d'une EDPS conduite par un bruit poissonnien. \textit{Probab. Th. Rel. Fields} \textbf{111}, 287--321 (1998)

\bibitem{San05} Sanz-Sol\'{e}, M.: Malliavin calculus with applications to stochastic partial differential equations. EPFL Press, Lausanne (2005)

\bibitem{SaS02} Sanz-Sol\'{e}, M. and Sarr\`{a}, M.: H\"{o}lder continuity for the stochastic heat equation with spatially correlated noise. In \textit{Seminar on Stochastic Analysis, Random Fields and Applications, III (Ascona, 1999). Progr. Probab.} \textbf{52} 259--268. Birkh\"{a}user, Basel. (2002)


\bibitem{Wal86} Walsh, J.B.: An introduction to stochastic partial differential equations. Ecole d'Et\'{e} de Probabilit\'{e}s de Saint-Flour XIV. Lect. Notes in Math., vol. 1180, 266--437. Springer, Heidelberg (1986)

\bibitem{Wat84} Watanabe, S.: Lectures on stochastic differential equations and Malliavin calculus. In: Tata Institute of Fundamental Research Lectures on Mathematics and Physics, vol. 73. Springer, Berlin (1984)

\end{thebibliography}
\end{document}